\documentclass[reqno,11p]{article}
\usepackage{amsfonts}
\usepackage{pifont}
\usepackage{latexsym,amsmath,amssymb,cite,amsthm,mathrsfs}
\usepackage{xcolor}

\usepackage[pdfborder={0 0 0}]{hyperref}

\numberwithin{equation}{section}
\theoremstyle{plain}
\newtheorem{theorem}{Theorem}[section]

\newtheorem{lemma}[theorem]{Lemma}
\newtheorem{proposition}[theorem]{Proposition}
\theoremstyle{definition}
\newtheorem{definition}[theorem]{Definition}
\theoremstyle{remark}

\newtheorem{remark}[theorem]{Remark}

\newcommand{\s}{{\rm S}}
\newcommand{\lip}{{\rm lip}}

\newcommand{\Leb}{{\rm Leb}}
\newcommand{\RCD}{{\sf RCD}}

\newcommand{\N}{\mathbb{N}}
\newcommand{\Q}{\mathbb{Q}}
\newcommand{\R}{\mathbb{R}}

\newcommand{\mm}{{\mbox{\boldmath$m$}}}

\newcommand{\ggamma}{{\mbox{\boldmath$\gamma$}}}

\newcommand{\mmu}{{\mbox{\boldmath$\mu$}}}

\newcommand{\ppi}{{\mbox{\boldmath$\pi$}}}

\newcommand{\sfd}{{\sf d}}

\newcommand{\Kliminf}{K\kern-3pt-\kern-2pt\mathop{\rm lim\,inf}\limits}  
\newcommand{\supp}{\mathop{\rm supp}\nolimits}   
\newcommand{\Lip}{\mathop{\rm Lip}\nolimits}          
\renewcommand{\d}{{\mathrm d}}

\newcommand{\restr}[1]{\lower3pt\hbox{$|_{#1}$}}

\newcommand{\eps}{\varepsilon}  
\newcommand{\nchi}{{\raise.3ex\hbox{$\chi$}}}

\newcommand{\limi}{\varliminf}
\newcommand{\lims}{\varlimsup}

\newcommand{\prob}[1]{\mathscr P(#1)}                   
\newcommand{\probt}[1]{\mathscr P_2(#1)}                   
\newcommand{\e}{{\rm{e}}}                           

\renewcommand{\mm}{\mathfrak m}                                

\newcommand{\weakgrad}[1]{|\nabla #1|_w} 

\renewcommand{\weakgrad}[1]{|D#1|}

\def\Xint#1{\mathchoice {\XXint\displaystyle\textstyle{#1}}
{\XXint\textstyle\scriptstyle{#1}}
{\XXint\scriptstyle\scriptscriptstyle{#1}}
 {\XXint\scriptscriptstyle\scriptscriptstyle{#1}}
  \!\int}
\def\XXint#1#2#3{{\setbox0=\hbox{$#1{#2#3}{\int}$} \vcenter{\hbox{$#2#3$}}\kern-.5\wd0}}

  \def\dashint{\Xint-}

\DeclareMathOperator*{\esssup}{ess\,sup}

\title{The  continuity equation on metric measure spaces}

\allowdisplaybreaks[4]

\begin{document}
\author{ Nicola Gigli\thanks{
 \textit{Previous address:} University of Nice
\textit{Current address:} Institut de math\'ematiques de Jussieu, UPMC,  \textsf{nicola.gigli@imj-prg.fr}}
\and Bang-Xian Han \thanks
{ Universit\'{e} Paris Dauphine}
}

\maketitle

\begin{abstract}
Aim of this paper is to show that it makes sense to write the continuity equation on a metric measure space $(X,\sfd,\mm)$ and that absolutely continuous curves $(\mu_t)$ w.r.t. the distance $W_2$ can be completely characterized as solutions of the continuity equation itself, provided we impose the condition  $\mu_t\leq C\mm$ for every $t$ and some $C>0$.

\end{abstract}

\tableofcontents

\section{Introduction}

A crucial intuition of Otto \cite{Otto01}, inspired by the work of Benamou-Brenier \cite{BenamouBrenier00}, has been to realize that absolutely continuous curves of measures $(\mu_t)$ w.r.t. the quadratic transportation distance $W_2$ on $\R^d$ can be interpreted as solutions of the continuity equation
\begin{equation}
\label{eq:contintro}
\partial_t\mu_t+\nabla\cdot(v_t\mu_t)=0,
\end{equation}
where the vector fields $v_t$ should be considered as the `velocity' of the moving mass $\mu_t$ and, for curves with square-integrable speed, satisfy
\begin{equation}
\label{eq:normvt}
\int_0^1\int|v_t|^2\,\d\mu_t\,\d t<\infty.
\end{equation}
This intuition has been made rigorous by the first author, Ambrosio and Savar\'e in \cite{AmbrosioGigliSavare08}, where it has been used to develop a solid first order calculus on the space $(\probt{\R^d},W_2)$, with particular focus on  the study of gradient flows.

Heuristically speaking, the continuity equation describes the link existing between the `vertical derivative' $\partial_t\mu_t$ (think to it as variation of the densities, for instance) and the `horizontal displacement' $v_t$. In this sense it provides the crucial link between analysis made on the $L^p$ spaces, where the distance is measured `vertically',  and the one based on optimal transportation, where distances are measured by `horizontal' displacement. This is indeed the heart of the crucial substitution made by Otto in \cite{Otto01} who, to define the metric tensor $g_\mu$ on the space $(\probt{\R^d},W_2)$ at a measure $\mu=\rho\mathcal L^d$ considers a `vertical' variation $\delta\rho$ such that $\int\delta\rho\,\d\mathcal L^d=0$, then looks for solutions of
\begin{equation}
\label{eq:otto}
\delta\rho=-\nabla\cdot(\nabla\varphi\rho),
\end{equation}
and finally defines
\begin{equation}
\label{eq:otto2}
g_\mu(\delta\rho,\delta\rho):=\int|\nabla\varphi|^2\,\d\mu.
\end{equation}
The substitution \eqref{eq:otto} is then another way of thinking at the continuity equation, while the definition \eqref{eq:otto2} corresponds to the integrability requirement \eqref{eq:normvt}.

It is therefore not surprising that each time one wants to put in relation the geometry of optimal transport with that of $L^p$ spaces some form of continuity equation must be studied. In the context of analysis on non-smooth structures, this has been implicitly done in \cite{Gigli-Kuwada-Ohta10,AmbrosioGigliSavare11} to show that the gradient flow of the relative entropy on the space $(\probt X,W_2)$ produces the same evolution of the gradient flow of the energy (sometime called Cheeger energy or Dirichlet energy) in the space $L^2(X,\mm)$, where $(X,\sfd,\mm)$ is some given metric measure space.

\bigskip

The purpose of this paper is to make these arguments more explicit and to show that:
\begin{itemize}
\item[i)] It is possible to formulate the continuity equation on general metric measure spaces $(X,\sfd,\mm)$,
\item[ii)] Solutions of the continuity equation completely characterize absolutely continuous curves $(\mu_t)\subset\probt X$ with square-integrable speed w.r.t. $W_2$ and such that $\mu_t\leq C\mm$ for every $t\in[0,1]$ and some $C>0$.
\end{itemize}
In fact, the techniques we use can directly produce similar results for the distances $W_p$, $p\in(1,\infty)$, and for curves whose speed is in $L^1$ rather then in some $L^p$, $p>1$. Yet, we prefer not to discuss the full generality in order to concentrate on the main ideas.

\bigskip

Let us discuss how to formulate the  continuity equation on a metric measure space where no a priori smooth structure is available. Notice that in the smooth setting  \eqref{eq:contintro} has to be understood in the sense of distributions.  If we assume weak continuity of $(\mu_t)$, this is equivalently formulated as the fact that that for every $f\in C^\infty_c(\R^d)$ the map $t\mapsto\int f\,\d\mu_t$ is absolutely continuous and the identity
\[
\frac\d{\d t}\int f\,\d\mu_t=\int \d f(v_t)\,\d\mu_t,
\]
holds for a.e. $t\in[0,1]$. In other words, the vector fields $v_t$ only act on differential of smooth functions and can therefore be thought of as linear functionals $L_t$ from the space of differentials of smooth functions to $\R$. Recalling \eqref{eq:normvt}, the norm $\|L_t\|_{\mu_t}^*$ of $L_t$ should be defined as
\[
\frac12(\|L_t\|^*_{\mu_t})^2=\sup_{f\in C^\infty_c(\R^d)}L_t(f)- \frac12\int |\d f|^2\,\d\mu_t,
\]
so that being $(\mu_t)$ 2-absolutely continuous is equivalent to require  that $t\mapsto\|L_t\|^*_{\mu_t}\in L^2(0,1)$.

Seeing the continuity equation in this way allows for a formulation of it  in the abstract context of metric measure spaces $(X,\sfd,\mm)$. Indeed, recall that there is a well established notion of `space of functions having distributional differential in $L^2(X,\mm)$', which we will denote by $\s^2(X)=\s^2(X,\sfd,\mm)$ and that for each function $f\in\s^2(X)$ it is well defined the `modulus of the distributional differential $|Df|\in L^2(X,\mm)$'.

Then given a linear map $L:\s^2(X)\to\R$ and $\mu$ such that $\mu\leq C\mm$ for some $C>0$ we can define the norm $\|L\|_\mu^*$ as
\begin{equation}
\label{eq:intronorm}
\frac12(\|L\|_\mu^*)^2:=\sup_{f\in\s^2(X)}L(f)-\frac12\int |Df|^2\,\d\mu.
\end{equation}
Hence given a curve $(\mu_t)\subset\probt X$ such that $\mu_t\leq C\mm$ for some $C>0$ and every $t\in[0,1]$ and a family $\{L_t\}_{t\in[0,1]}$ of maps from $\s^2(X)$ to $\R$ such that $\int_0^1(\|L_t\|^*_{\mu_t})^2\,\d t<\infty$, we can say that the curve $(\mu_t)\subset\probt X$ solves the continuity equation
\[
\partial_t\mu_t=L_t,
\]
provided:
\begin{itemize}
\item[i)] for every $f\in\s^2(X)$ the map $t\mapsto\int f\,\d\mu_t$ is absolutely continuous,
\item[ii)] the identity
\begin{equation}
\label{eq:contintro3}
\frac{\d }{\d t}\int f\,\d\mu_t=L_t(f),
\end{equation}
holds for a.e. $t$.
\end{itemize}
Then we show  that such formulation of the continuity equation fully characterizes absolutely continuous curves $(\mu_t)$ with square-integrable speed on the space $(\probt X,W_2)$, provided we restrict the attention to curves such that  $\mu_t\leq C\mm$ for some $C>0$ and every $t\in[0,1]$. See Theorem \ref{thm:main}.

Concerning the proof of this result, we remark that the implication from absolute continuity of $(\mu_t)$ to the `PDE' \eqref{eq:contintro3} is quite easy to establish and follows essentially from the definition of Sobolev functions. This is the easy implication even in the smooth context whose proof carries over quite smoothly to the abstract setting, the major technical difference being that we don't know if in general the space $\s^2(X)$ is separable or not, a fact which causes some complications in the way we can really write down the equation \eqref{eq:contintro3}, see Definition \ref{def:solcont}.

The converse one is more difficult, as it amounts in proving that the differential identity \eqref{eq:contintro3} is strong enough to guarantee absolute continuity of the curve. The method used in the Euclidean context consists in regularizing the curve, applying the Cauchy-Lipschitz theory to the approximating sequence to find a flow of the approximating vector fields which can be used to transport $\mu_t$ to $\mu_s$ and finally in passing to the limit. By nature, this approach cannot be used in non-smooth situations. Instead, we use a crucial idea due to Kuwada which has already been applied to study the heat flow \cite{Gigli-Kuwada-Ohta10,AmbrosioGigliSavare11}. It amounts in passing to the dual formulation of the optimal transport problem by noticing that
\begin{equation}
\label{eq:dualintro}
\frac12W_2^2(\mu_{1},\mu_0)=\sup \int Q_1\varphi\,\d\mu_{1}-\int \varphi\,\d\mu_0,
\end{equation}
the $\sup$ being taken among all Lipschitz and bounded $\varphi:X\to\R$, where $Q_t\varphi$ is the evolution of $\varphi$ via the Hopf-Lax formula. A general result obtained in \cite{AmbrosioGigliSavare11} has been that it holds
\begin{equation}
\label{eq:HLintro}
\frac{\d }{\d t}Q_t\varphi(x)+\frac{\lip(Q_t\varphi)^2(x)}2\leq 0,
\end{equation}
for every $t$ except a countable number, where $\lip(f)$ is the local Lipschitz constant $f$. Thus  we can formally write
\[
\begin{split}
 \int Q_1\varphi\,\d\mu_{1}-\int \varphi\,\d\mu_0&=\int_0^1\frac{\d}{\d t}\int Q_t\varphi\,\d\mu_t\,\d t\\
\text{by }\eqref{eq:contintro3}\qquad&=\int_0^1\int\frac{\d}{\d t}Q_t\varphi\,\d\mu_t\,\d t+\int_0^1L_t(Q_t\varphi)\,\d t,\\
\text{by }\eqref{eq:intronorm},\eqref{eq:HLintro}\qquad&\leq \int_0^1-\frac{\lip(Q_t\varphi)^2}{2}\,\d\mu_t\,\d t+\frac12\int_0^1(\|L_t\|_{\mu_t}^*)^2\,\d t+\frac12\int_0^1\int|DQ_t\varphi|^2\,\d\mu_t\,\d t.
\end{split}
\]
Using the fact that $|D f|\leq \lip(f)$ $\mm$-a.e. for every Lipschitz $f$ we then conclude that
\[
 \int Q_1\varphi\,\d\mu_{1}-\int \varphi\,\d\mu_0\leq \frac12\int_0^1(\|L_t\|_{\mu_t}^*)^2\,\d t.
 \]
Here the right hand side does not depend on $\varphi$, hence by \eqref{eq:dualintro} we deduce
\[
W_2^2(\mu_1,\mu_0)\leq \int_0^1(\|L_t\|_{\mu_t}^*)^2\,\d t,
\]
which bounds  $W_2$ in terms of the $L_t$'s only. Replacing $0,1$ with general $t,s\in[0,1]$ we deduce the desired absolute continuity. As presented here, the computation is only formal, but a rigorous justification can be given, thus leading to the result. See the proof of Theorem \ref{thm:main}.

It is worth pointing out that Kuwada's lemma works even if we don't know any uniqueness result for the initial value problem \eqref{eq:contintro3}. That is we don't know if given $\mu_0$ and a family of operators $L_t$ from $\s^2(X)$ to $\R$  the solution of \eqref{eq:contintro3} is unique or not, because we ``can't follow the flow of the $L_t$'s''. Yet, it is possible to deduce anyway that any solution is absolutely continuous.

It is also worth to make some comments about the assumption $\mu_t\leq C\mm$. Notice that if we don't impose any condition on the $\mu_t$'s, we could consider curves of the kind $t\mapsto\delta_{\gamma_t}$, where $\gamma$ is a given Lipschitz curve. In the smooth setting we see that such curve solves
\[
\partial_t\delta_{\gamma_t}+\nabla\cdot(\gamma'_t\delta_{\gamma_t})=0,
\]
so that to write the continuity equation for such curve amounts to know the value of $\gamma_t'$ at least for a.e. $t$. In the non-smooth setting to do this would mean to know who is the tangent space at $\gamma_t$ for a.e. $t$ along a Lipschitz curve $\gamma$,  an information which without any assumption on $X$ seems quite  too strong.  Instead, the process of considering only measures with bounded density has the effect of somehow `averaging out the unsmoothness of the space' and allows for the possibility of building a working differential calculus, a point raised and heavily used in \cite{Gigli13}. Here as application of the continuity equation to differential calculus we provide a Benamou-Brenier formula and describe the derivative of $\frac12W_2^2(\cdot,\nu)$ along an absolutely continuous curve.

\bigskip

We then study situations where the operators $L_t$ can be given somehow more explicitly. Recall that on the Euclidean setting the optimal (in the sense of energy-minimizer) vector fields $v_t$ appearing in \eqref{eq:contintro} always belong to the $L^2(\mu_t)$-closure of the set of gradients of smooth functions and that is some case they are really gradient of functions, so that \eqref{eq:contintro} can be written as
\begin{equation}
\label{eq:contintro2}
\partial_t\mu_t+\nabla\cdot(\nabla\phi_t\mu_t)=0,
\end{equation}
for some given smooth $\{\phi_t\}_{t\in[0,1]}$, which means that for $f$ smooth it holds
\[
\frac{\d}{\d t}\int f\,\d\mu_t=\int \d f(\nabla\phi_t)\,\d\mu_t.
\]
To interpret the equation \eqref{eq:contintro2} in the abstract framework we need to understand the duality relation between differentials and gradients of Sobolev functions on metric measure spaces. This has been done in \cite{Gigli12}, where for given $f,g\in\s^2(X)$ the two functions $D^-f(\nabla g)$ and $D^+f(\nabla g)$ have been introduced. If the space is a Riemannian manifold or a Finsler one with norms strictly convex, then we have $D^-f(\nabla g)=D^+(\nabla g)$ a.e. for every $f,g$, these being equal to the value of the differential of $f$ applied to the gradient of $g$ obtained by standard means. In the general case we do not have such single-valued duality, due to the fact that even in a flat normed situation the gradient of a function is not uniquely defined should  the norm be not strictly convex. Thus the best we can do is to define $D^-f(\nabla g)$ and $D^+f(\nabla g)$ as being, in a sense, the minimal and maximal value of the differential of $f$ applied to all the gradients of $g$.

Then we can interpret \eqref{eq:contintro2} in the non-smooth situation by requiring that for $f\in\s^2(X)$ it holds
\[
\int D^-f(\nabla\phi_t)\,\d\mu_t\leq \frac\d{\d t}\int f\,\d\mu_t\leq \int D^+f(\nabla\phi_t)\,\d\mu_t,\qquad a.e.\ t,
\]
and it turns out that this way of writing the continuity equation, which requires two inequalities rather than an equality, is still sufficient to grant absolute continuity of the curve.

Notice that in the Euclidean setting, if the functions $\phi_t$ are smooth enough we can construct the flow associated to $\nabla\phi_t$ by solving
\[
\left\{\begin{array}{ll}
\displaystyle{\frac\d{\d t}T(x,t,s)}&\!\!\!=\nabla\phi_t(T(x,t,s)),\\
\\
T(x,t,t)&\!\!\!=x,
\end{array}
\right.
\]
so that the curves $t\mapsto T(x,t,s)$ are gradient flows of the evolving function $\phi_t$ and a curve $(\mu_t)$ solves \eqref{eq:contintro3} if and only if $\mu_t= T(\cdot,t,0)_\sharp\mu_0$ for every $t\in[0,1]$. Interestingly enough, this point of view can be made rigorous even in the setting of metric measure spaces and a similar characterization of solutions of \eqref{eq:contintro3} can be given, see Theorem \ref{thm:contgrad}.

\bigskip

We conclude the paper by showing that  the heat flows and the geodesics satisfy the same sort of continuity equation they satisfy in the smooth case, namely
\[
\partial_t\mu_t+\nabla\cdot(\nabla(-\log(\rho_t))\mu_t)=0,
\]
for the heat flow, where $\mu_t=\rho_t\mm$, and
\[
\partial_t\mu_t+\nabla\cdot(\nabla\phi_t\mu_t)=0,
\]
with $\phi_t=-Q_{1-t}(-\varphi^c)$ for the geodesics, where $\varphi$ is a Kantorovich potential inducing the geodesic itself. Here the aim is not to prove new results, as these two examples were already considered in the literature \cite{Gigli-Kuwada-Ohta10,AmbrosioGigliSavare11,AmbrosioGigliSavare11-2,Gigli13}, but rather to show that they are compatible with the theory we propose. We also discuss in which sense and under which circumstances an heat flow and a geodesic can be considered not just as absolutely continuous curves on $(\probt X,W_2)$, but rather as $C^1$ curves.

\section{Preliminaries}

\subsection{Metric spaces and optimal transport}

We quickly recall here those basic fact about analysis in metric spaces and optimal transport we are going to use in the following. Standard references are \cite{AmbrosioGigliSavare08}, \cite{Villani09} and \cite{AmbrosioGigli11}.

Let $(X,\sfd)$ be a metric space. Given $f:X\to\R$ the local Lipschitz constant $\lip(f):X\to[0,\infty]$ is defined as
\[
\lip(f)(x):=\lims_{y\to x}\frac{|f(y)-f(x)|}{\sfd(x,y)},
\]
if $x$ is not isolated and 0 otherwise. Recall that the Lipschitz constant of Lipschitz function is defined as:
\[
\Lip(f):= \mathop{\sup}_{x \neq y} \frac{|f(y)-f(x)|}{\sfd(x,y)}.
\]
 In particular, if $(X,\sfd)$ is a geodesic space, we have $\Lip(f)=\mathop{\sup}_x \lip(f)(x)$.

A curve $\gamma:[0,1]\to X$ is said absolutely continuous provided there exists $f\in L^1(0,1)$ such that
\begin{equation}
\label{eq:defac}
\sfd(\gamma_s,\gamma_t) \leq \int_t^s f(r) \,\d r,\qquad\forall t,s \in [0,1], \ t<s.
\end{equation}
For an absolutely continuous curve $\gamma$ it can be proved that the limit
\[
\lim_{h\to 0}\frac{\sfd(\gamma_{t+h},\gamma_t)}{|h|},
\]
exists for a.e. $t$ and thus defines a function, called metric speed and denoted by $|\dot\gamma_t|$, which is in $L^1(0,1)$ and is minimal, in the a.e. sense, among the class of $L^1$-functions $f$ for which \eqref{eq:defac} holds.

If there exists $f\in L^2(0,1)$ for which \eqref{eq:defac} holds, we say that the curve is 2-absolutely continuous (2-a.c. in short). In the following we will often write $\int_0^1|\dot\gamma_t|^2\,\d t$ for a curve $\gamma$ which a priori is only continuous: in this case the value of the integral is taken by definition $+\infty$ if $\gamma$ is not absolutely continuous.

\bigskip

The space of continuous curves on $[0,1]$ with values in $X$ will be denoted by $C([0,1],X)$ and equipped with the $\sup$ distance. Notice that if $(X,\sfd)$ is complete and separable, then $C([0,1],X)$ is complete and separable as well. For $t\in[0,1]$ we denote by $\e_t:C([0,1],X)\to X$ the \emph{evaluation map} defined by
\[
\e_t(\gamma):=\gamma_t,\qquad\forall \gamma\in C([0,1],X).
\]
For $t,s\in[0,1]$ the map ${\rm restr}_t^s$ from $C([0,1],X)$ to itself is given by
\[
({\rm restr}_t^s\gamma)_r:=\gamma_{t+r(s-t)},\qquad\forall \gamma\in C([0,1],X).
\]

The set of Borel probability measures on $X$ is denoted by $\prob X$ and $\probt X\subset\prob X$ is the space of probability measures with finite second moment. We equip $\probt X$ with the quadratic transportation distance $W_2$ defined by
\begin{equation}
\label{eq:defw2}
W_2^2(\mu,\nu):=\inf\int \sfd^2(x,y)\,\d\ggamma(x,y),
\end{equation}
the $\inf$ being taken among all $\ggamma\in\prob{X^2}$ such that
\[
\begin{split}
\pi^1_\sharp\ggamma&=\mu,\\
\pi^2_\sharp\ggamma&=\nu.
\end{split}
\]
Given $\varphi:X\to\R\cup\{-\infty\}$ not identically $-\infty$ the $c$-transform  $\varphi^c:X\to\R\cup\{-\infty\}$ is defined by
\[
\varphi^c(y):=\inf_{x\in X}\frac{\sfd^2(x,y)}2-\varphi(x).
\]
$\varphi$ is said $c$-concave provided it is not identically $-\infty$ and $\varphi=\psi^c$ for some $\psi:X\to\R\cup\{-\infty\}$. Equivalently, $\varphi$ is $c$-concave if  it is not identically $-\infty$ and $\varphi^{cc}=\varphi$. Given a $c$-concave function $\varphi$, its $c$-superdifferential $\partial^c\varphi\subset X^2$ is defined as the set of $(x,y)$ such that
\[
\varphi(x)+\varphi^c(y)=\frac{\sfd^2(x,y)}{2},
\]
and for $x\in X$ the set $\partial^c\varphi(x)$ is the set of $y$'s such that $(x,y)\in\partial^c\varphi$. Notice that for general $(x,y)\in X^2$ we have $\varphi(x)+\varphi^c(y)\leq\frac{\sfd^2(x,y)}{2}$, thus $y\in\partial^c\varphi(x)$ can be equivalently formulated as
\[
\varphi(z)-\varphi(x)\leq\frac{\sfd^2(z,y)}{2}-\frac{\sfd^2(x,y)}{2},\qquad\forall z\in X.
\]
It turns out that for $\mu,\nu\in\probt X$ the distance $W_2(\mu,\nu)$ can be found as maximization of the dual problem of the optimal transport \eqref{eq:defw2}:
\begin{equation}
\label{eq:dual}
\frac12W_2^2(\mu,\nu)=\sup\int \varphi\,\d\mu+\int\varphi^c\,\d\nu,
\end{equation}
the $\sup$ being taken among all $c$-concave functions $\varphi$. Notice that the integrals in the right hand side are well posed because for $\varphi$ $c$-concave and $\mu,\nu\in\probt X$ we always have $\max\{\varphi,0\}\in L^1(\mu)$ and $\max\{\varphi^c,0\}\in L^1(\nu)$. The $\sup$ is always achieved and any maximizing  $\varphi$ is called Kantorovich potential from $\mu$ to $\nu$. For any Kantorovich potential we have in particular $\varphi\in L^1(\mu)$ and $\varphi^c\in L^1(\nu)$. Equivalently, the $\sup$ in \eqref{eq:dual} can be taken among all $\varphi:X\to \R$ Lipschitz and bounded.

\bigskip

We shall make frequently use of the following superposition principle, proved in \cite{Lisini07} (see also the original argument in the Euclidean framework \cite{AmbrosioGigliSavare08}):
\begin{proposition}
Let $(\mu_t)\subset\probt X$ be a 2-a.c. curve w.r.t. $W_2$. Then there exists $\ppi\in\prob{C([0,1],X)}$ such that
\[
\begin{split}
(\e_t)_\sharp\ppi&=\mu_t, \qquad \qquad\forall t\in[0,1],\\
\int_0^1|\dot\mu_t|^2\,\d t&=\iint_0^1|\dot\gamma_t|^2\,\d t\,\d \ppi(\gamma),
\end{split}
\]
and in particular $\ppi$ is concentrated on the set of 2-a.c. curves on $X$. For any such $\ppi$ we also have
\[
|\dot\mu_t|^2=\int|\dot\gamma_t|^2\,\d\ppi(\gamma),\qquad a.e.\ t\in[0,1].
\]
\end{proposition}
Any plan $\ppi$ associated to the curve $(\mu_t)$ as in the above proposition will be called \emph{lifting} of $(\mu_t)$.

\subsection{Metric measure spaces and Sobolev functions}
Spaces of interest for this paper are metric measure spaces $(X,\sfd,\mm)$ which will always be assumed to satisfy:
\begin{itemize}
\item $(X,\sfd)$ is  complete and separable,
\item the measure $\mm$ is a non-negative and non-zero Radon measure on $X$.
\end{itemize}
We shall consider the structure $(X,\sfd,\mm)$ as given and often omit to highlight the explicit dependence on the distance and the measure of our constructions. For instance, we shall often denote by $W^{1,2}(X)$ the Sobolev space of real valued functions defined on $X$ (see below).

Given a curve $(\mu_t)\subset \prob X$ we shall say that it has \emph{bounded compression} provided there is $C>0$ such that $\mu_t\leq C\mm$ for every $t\in[0,1]$. Similarly, given $\ppi\in\prob{C([0,1],X)}$ we shall say that it has bounded compression provided $t\mapsto\mu_t:=(\e_t)_\sharp\ppi$ has bounded compression.

We shall now recall the definition of Sobolev functions `having distributional differential in $L^2(X,\mm)$'. The definition we adopt comes from \cite{Gigli12} which in turn is a reformulation of the one proposed in \cite{AmbrosioGigliSavare11}. For the proof that this approach produces the same concept as the one discussed in \cite{Heinonen07} and its references, see \cite{AmbrosioGigliSavare11}.

\begin{definition}[Test plans]
Let $(X,\sfd,\mm)$ be a metric measure space and $\ppi \in\prob{C([0,1],X)}$. We say  that $\ppi$ is a test plan provided it has bounded compression and
\[
\iint_0^1 |\dot{\gamma}_t|^2 \,\d t\,\d\ppi(\gamma) < +\infty.
\]
\end{definition}

\begin{definition}[The Sobolev class $\s^2(X)$]
Let $(X,\sfd,\mm)$ be a metric measure space. The Sobolev class $\s^2(X)=\s^2(X,\sfd,\mm)$ is the space of all Borel functions  $f : X \rightarrow \mathbb{R}$ such that  there exists a function $G\in L^2(X,\mm)$, $G\geq 0$ such that:
\[
\int |f(\gamma_1)- f(\gamma_0)|\, \d\ppi(\gamma) \leq \iint_0^1 G(\gamma_t)|\dot{\gamma}_t|\, \d t\, \d\ppi(\gamma),
\]
for every test plan $\ppi$. In this case, $G$ is called a weak upper gradient of $f$.
\end{definition}

It can be proved that for $f\in\s^2(X)$ there exists a minimal, in the $\mm$-a.e. sense, weak upper gradient: we shall denote it by $|Df|$.

Basic calculus rules for $|Df|$ are the following, all the expressions being intended $\mm$-a.e.:

\noindent\underline{Locality} For every $f,g\in\s^2(X)$ we have
\begin{align}
\label{eq:nullgrad}
\weakgrad f&=0,\qquad&&\textrm{on }f^{-1}(N),\qquad\forall N\subset \R\textrm{, Borel with }\mathcal L^1(N)=0,\\
\label{eq:localgrad}
\weakgrad f&=\weakgrad g,\qquad&&\mm-a.e. \ on \ \{f=g\}.
\end{align}

\noindent\underline{Weak gradients and local Lipschitz constants}. For any $f:X\to\R$ locally Lipschitz it holds
\begin{equation}
\label{eq:lipweak}
\weakgrad f\leq \lip(f).
\end{equation}
\noindent\underline{Vector space structure}. $\s^2(X)$ is a vector space and for every $f,g\in\s^2(X)$, $\alpha,\beta\in\R$ we have
\begin{equation}
\label{eq:vectorstru}
\weakgrad{(\alpha f+\beta g)}\leq |\alpha|\weakgrad f+|\beta|\weakgrad g.
\end{equation}

\noindent\underline{Algebra structure}.  $L^\infty\cap\s^2(X)$ is an algebra and for every $f,g\in L^\infty\cap\s^2(X)$ we have
\begin{equation}
\label{eq:leibbase}
\weakgrad{(fg)}\leq |f|\weakgrad g+|g|\weakgrad f.
\end{equation}
Similarly, if $f\in \s^2(X)$ and $g$ is  Lipschitz and bounded, then $fg\in\s^2(X)$ and the bound \eqref{eq:leibbase} holds.

\noindent\underline{Chain rule}. Let  $f\in \s^2(X)$ and $\varphi:\R\to \R$ Lipschitz. Then $\varphi\circ f\in \s^2(X)$ and
\begin{equation}
\label{eq:chainbase}
\weakgrad{(\varphi\circ f)}=|\varphi'|\circ f\weakgrad f,
\end{equation}
where $|\varphi'|\circ f$ is defined arbitrarily at points where $\varphi$ is not differentiable (observe that the identity \eqref{eq:nullgrad} ensures that on $f^{-1}(\mathcal N)$ both $\weakgrad{(\varphi\circ f)}$ and $\weakgrad f$ are 0 $ \mm$-a.e., $\mathcal N$ being the negligible set of points of non-differentiability of $\varphi$).

We equip $\s^2(X)$ with the seminorm $\|f\|_{\s^2}:=\||Df|\|_{L^2(X)}$ { and introduce the Sobolev space $W^{1,2}(X)=W^{1,2}(X,\sfd,\mm)$ as $W^{1,2}(X):=L^2\cap\s^2(X)$ equipped with the norm 
\[
\|f\|_{W^{1,2}}^2:=\|f\|^2_{L^2}+\|f \|^2_{\s^2}.
\]
We recall that $W^{1,2}(X)$ is a Banach space.}
In the following, we will sometimes need to work with spaces $(X,\sfd,\mm)$ such that $\s^2(X)$ is separable, we thus recall the following general criterion:
\begin{proposition}\label{prop:separ}
Let $(X,\sfd,\mm)$  be a metric measure space with $\mm$ giving finite mass to bounded sets. Assume that   $W^{1,2}(X,\sfd,\mm)$  is reflexive. Then $\s^2(X)$ is separable.

In particular, let $(X,\sfd,\mm)$ be a metric measure space satisfying  one of the following properties:
\begin{itemize} 
\item[i)] $(X,\sfd)$ is doubling, i.e. there is $N\in\N$ such that for any $r>0$ any ball of radius $2r$ can be covered by $N$ balls of radius $r$. 
\item[ii)] The seminorm $\|\cdot\|_{\s^2}$ satisfies the parallelogram rule, and $\mm$ gives finite mass to bounded sets.
\end{itemize}
Then $\s^2(X)$ is separable.
\end{proposition}
\begin{proof}
In \cite{AmbrosioColomboDimarino12} it has been proved that if $W^{1,2}(X)$ is reflexive, then it is separable.

Thus to conclude it is sufficient to show that if $W^{1,2}(X)$ is separable and $\mm$ gives finite mass to bounded sets (this being trivially true in the case $(i)$), then $\s^2(X)$ is separable as well. To this aim, let $f\in\s^2(X)$, consider the truncated functions $f_n:=\min\{n,\max\{-n,f\}\}$ and notice that thanks to \eqref{eq:chains} we have $\|f_n-f\|_{\s^2}\to 0$ as $n\to\infty$. Thus we can reduce to consider the case of $f\in L^\infty\cap\s^2(X)$. Let $B_n\subset X$ be a nondecreasing sequence of  bounded sets covering $X$, and for each $n\in\N$, $\nchi_n:X\to[0,1]$ a 1-Lipschitz function with bounded support and identically 1 on $B_n$. For $f\in L^\infty\cap \s^2(X)$, by \eqref{eq:leibbase} we have $f\nchi_n\in L^\infty\cap \s^2(X)$ as well and furthermore $\supp(\nchi_nf)$ is bounded.  Given that $\mm$ gives finite mass to bounded sets we deduce that $f\nchi_n\in W^{1,2}(X)$, and the locality property \eqref{eq:localgrad} ensures that $\|\nchi_nf-f\|_{\s^2}\to 0$ as $n\to\infty$.

At last, still in \cite{AmbrosioColomboDimarino12}, it has been shown that if $(X,\sfd)$ is doubling, then $W^{1,2}(X)$ is reflexive. On the other hand, if $(ii)$ holds, then it is obvious that  $W^{1,2}(X)$ is Hilbert, and hence reflexive. Therefore $\s^2(X)$ is separable if $(i)$ or $(ii)$ holds.
\end{proof}

\subsection{Hopf-Lax formula and Hamilton-Jacobi equation}
Here we recall the main properties of the Hopf-Lax formula and its link with the Hamilton-Jacobi equation in a metric setting. For a proof of these results see \cite{AmbrosioGigliSavare11}.

\begin{definition}[Hopf-Lax formula]
Given $f: X \to \mathbb{R}$ a function and $t>0$ we define  $Q_tf : X \to \mathbb{R}\cup\{-\infty\}$ as
\[
Q_tf(x):=\inf_{y\in X} f(y)+\frac{\sfd^2(x,y)}2.
\]
We also put $Q_0 f:= f$.
\end{definition}

\begin{proposition}[Basic properties of the Hopf-Lax formula]\label{prop:HL}
Let $f:X\to\R$ be Lipschitz and bounded. Then the following hold.
\begin{itemize}
\item[i)] For every $t\geq 0$ we have $\Lip(Q_tf)\leq 2\Lip(f)$.
\item[ii)] For every $x\in X$ the map $[0,\infty)\ni t\mapsto Q_tf(x)$ is continuous, locally semiconcave on $(0,\infty)$ and the inequality
\[
\frac{\d}{\d t} Q_tf(x) +\frac{\lip(Q_tf)^2(x)}{2} \leq 0,
\]
holds for every $t\in(0,\infty)$ with at most a countable number of exceptions.
\item[iii)] The map $(0,\infty)\times X\ni (t,x)\mapsto \lip(Q_tf)(x)$ is upper-semicontinuous.
\end{itemize}
\end{proposition}

\section{The continuity equation $\partial_t\mu_t=L_t$}

\subsection{Some definitions and conventions}\label{se:def}

Let $\mu\in\probt X$ be such that $\mu\leq C\mm$ for some $C>0$. We define the seminorm $\|\cdot\|_\mu$ on $\s^2(X)$ as
\[
\|f\|_\mu^2:=\int|Df|^2\,\d\mu.
\]

\begin{definition}[The cotangent space ${\rm CoTan}_\mu(X)$]
For  $\mu\in\probt X$ with $\mu\leq C\mm$ for some $C>0$ consider the quotient space $\s^2(X)/\sim_\mu$, where $f\sim_\mu g$ if $\|f-g\|_\mu=0$.

The cotangent space  ${\rm CoTan}_\mu(X)$ is then defined as the completion of $\s^2(X)/\sim_\mu$ w.r.t. its natural norm. The norm on ${\rm CoTan}_\mu(X)$ will still be denoted by $\|\cdot\|_\mu$.
\end{definition}

Given a linear map $L:\s^2(X)\to\R$ and $\mu$ as above, we denote by $\|L\|_\mu^*\geq 0$ the quantity given by
\[
\frac12(\|L\|_\mu^*)^2:=\sup_{f\in\s^2(X)}L(f)-\frac12\|f\|_\mu^2.
\]
Linear operators   $L:\s^2(X)\to\R$  such that  $\|L\|_\mu^*<\infty$ are in 1-1 correspondence with elements of the dual of ${\rm CoTan}_\mu(X)$. Abusing a bit the notation, we will often identify such operators $L$ with the induced linear mapping on ${\rm CoTan}_\mu(X)$.

\subsection{A localization argument}
In this section $(\mu_t)\subset \probt X$ is a given $W_2$-continuous curve with bounded compression and we consider a functional $L:\s^2(X)\to L^1(0,1)$ satisfying the inequality
\[
\int_t^s L(f)(r)\,d r\leq \sqrt{\int_t^s G^2_r\,\d r}\sqrt{\int_t^s\|f\|_{\mu_r}^2\,\d r}
\]
{ for some}  $G\in L^2(0,1)$, for every $f\in\s^2(X)$ and $t,s\in[0,1]$, $t<s$. The question we address is up to what extent we can deduce that for such $L$ there are operators $L_t:\s^2(X)\to\R$ such that $L(f)(t)=L_t(f)$ for $\mathcal L^1$-a.e. $t\in[0,1]$. We will see in a moment that this is always the case in an appropriate sense, but to deal with the case of $\s^2(X)$ non separable we need to pay some attention to the set of Lebesgue points of $L(f)$.

Thus for given  $g\in L^1(0,1)$ we shall denote by $\Leb(g)\subset(0,1)$ the set of $t$'s such that {
\[
\begin{split}
\lims_{\eps_1,\eps_2\downarrow0}\dashint_{t-\eps_1}^{t+\eps_1}\dashint_{t-\eps_2}^{t+\eps_2}|g_s-g_r|\,\d s\,\d r=0.
\end{split}
\]
Clearly the set $\Leb(g)$ contains all the Lebesgue points of any representative of $g$ (notice that the usual definition of Lebesgue point is sensible to the chosen representative) and in particular we have $\mathcal L^1(\Leb(g))=1$. For $t\in\Leb(g)$ the limit of
\[
\dashint_{t-\eps}^{t+\eps}g_s\,\d s,
\]
as $\eps\downarrow0$ exists and is finite.  We shall denote its value by $\bar g$, so that $\bar g:\Leb(g)\to\R$ is a well chosen representative of $g$ everywhere defined on $\Leb(g)$.}

We then have the following result.
\begin{lemma}\label{le:local}
Let $(\mu_t)\subset\probt X$ be a $W_2$-continuous curve of bounded compression and  $L:\s^2(X)\to L^1(0,1)$ be a linear map such that for some $G\in L^2(0,1)$ the inequality
\begin{equation}\label{eq:forlocal}
\int_t^s L(f)(r)\,d r\leq \sqrt{\int_t^s G^2_r\,\d r}\sqrt{\int_t^s\|f\|_{\mu_r}^2\,\d r},\qquad \forall t,s\in[0,1],\ t<s,\ \forall f\in\s^2(X),
\end{equation}
holds.

Then there exists a family $\{{L}_t\}_{t\in[0,1]}$ of maps from $\s^2(X)$ to $\R$ such that for any $f\in\s^2(X)$ we have
\begin{align}
\label{eq:uniqueness}
L(f)(t)&=L_t(f),  &&a.e.\  t\in[0,1],\\
\label{eq:bound}
|L_t(f)|&\leq |G_t|\|f\|_{\mu_t}, &&a.e.\  t\in[0,1].
\end{align}
\end{lemma}

\begin{remark}
As a direct consequence of \eqref{eq:uniqueness}, if $\{{\tilde L}_t\}_{t\in[0,1]}$ is another family of maps satisfying the above, then for every $f\in \s^2(X)$ we have $L_t(f)=\tilde  L_t(f)$ for a.e. $t\in[0,1]$. 
\end{remark}

\begin{proof}

For $f\in \s^2(X)$ consider the set $\Leb({L}(f))$ and for $t\in(0,1) $ let ${V}_t\subset \s^2(X)$ be the set of $f$'s in $\s^2(X)$ such that $t\in \Leb({L}(f))$. The trivial inclusion
\[
\Leb(\alpha_1g_1+\alpha_2g_2)\supset\Leb(g_1)\cap\Leb(g_2),
\]
valid for any $g_1,g_2\in L^1(0,1)$ and $\alpha_1,\alpha_2\in\R$ and the linearity of $L$ grant that  ${V}_t$ is a vector space for every $t\in(0,1)$.

The $W_2$-continuity of $(\mu_t)$ grants in particular continuity w.r.t. convergence in duality with $C_b(X)$ and the further assumption that $\mu_t\leq C\mm$ for any $t\in[0,1]$ ensures continuity  w.r.t. convergence in duality with $L^1(X,\mm)$. Thus the map  $t\mapsto\int |Df|^2\,\d\mu_t$ is continuous for any $f\in \s^2(X)$. Hence from inequality \eqref{eq:forlocal} we deduce that for any $f\in\s^2(X)$ it holds
\[
|\overline{L(f)}(t)|\leq \sqrt{\overline{G^2_t}}\, \|f\|_{\mu_t},\qquad\forall t\in \Leb({L}(f))\cap \Leb(G^2),
\]
which we can rewrite as: for any $t\in \Leb(G^2)$ it holds
\[
|\overline{L(f)}(t)|\leq  \sqrt{\overline{G^2_t}}\, \|f\|_{\mu_t},\qquad\forall f\in V_t.
\]
In other words, for any $t\in \Leb(G^2)$ the map $V_t\ni f\mapsto L_t(f):=\overline{L(f)}(t)$ is a well defined linear map from $V_t$ to $\mathbb{R}$ with norm bounded by $ \sqrt{\overline{G^2_t}}$.

By Hahn-Banach we can extend this map to a map from $S^2$ to $\mathbb{R}$ with norm bounded by $G(t)$. Noticing that by construction we have $f\in \bar{V}_t$ for a.e. $t\in[0,1]$ for any $f\in S^2$, the family of maps $\bar{L}_t$ fulfill the thesis. To conclude notice that trivially it holds $ \sqrt{\overline{G^2_t}}=|G_t|$ for $\mathcal L^1$-a.e. $t$.
\end{proof}

\subsection{Main theorem}
{{We recall that given measurable maps $g_i:[0,1]\to \R\cup\{\pm\infty\}$ parametrized by  $i\in I$, where $I$ is a non-necessarily countable family of indexes, the essential supremum $\esssup_ig_i:[0,1]\to\R\cup\{\pm\infty\}$ is the (unique up to $\mathcal L$-a.e. equality) function $g$ such that
\[
\begin{split}
g&\geq g_i,\quad\mathcal L-a.e.\ \qquad\forall i\in I,\\
\tilde g&\geq g_i,\quad\mathcal L-a.e.\ \qquad\forall i\in I,\qquad\Rightarrow \qquad g\leq \tilde g,\quad\mathcal L-a.e..
\end{split}
\]}}

We start giving the definition of `distributional' solutions of the continuity equation in our setting:
\begin{definition}[Solutions of $\partial_t\mu_t=L_t$]\label{def:solcont}
Let $(X,\sfd,\mm)$ be a metric measure space, $(\mu_t)\subset\probt X$ a $W_2$-continuous curve with bounded compression and  $ \{L_t\}_{t\in[0,1]}$ a family of maps from $\s^2(X)$ to $\R$.

We say that $(\mu_t)$ solves the continuity equation
\begin{equation}
\label{eq:basecont}
\partial_t\mu_t=L_t,
\end{equation}
provided:
\begin{itemize}
\item[i)] for every $f\in\s^2(X)$ the map $t\mapsto L_t(f)$ is measurable and the map $N:[0,1]\to[0,\infty]$ defined by
\begin{equation}
\label{eq:normt}
\frac12N^2_t:=\esssup_{f\in\s^2(X)}L_t(f)-\frac12\|f\|_{\mu_t}^2,
\end{equation}
belongs to $L^2(0,1)$, i.e. for any $f$, $\frac12N^2_t \geq L_t(f)-\frac12\|f\|_{\mu_t}^2$  { for }a.e.\ $t$ and for any other $\bar{N}_t$ { having} this property, we have
$N_t \leq \bar{N}_t$ for a.e. $t$.

\item[ii)]  for every $f\in L^1\cap\s^2(X)$ the map $t\mapsto \int f\,\d\mu_t$ is in absolutely continuous and the identity
\[
\frac{\d}{\d t}\int f\,\d\mu_t=L_t(f),
\]
holds for a.e. $t$.
\end{itemize}
\end{definition}
Evidently,  if $\{\tilde L_t\}_{t\in[0,1]}$ is another family of maps such that $(\mu_t)$ solves $\partial_t\mu_t=\tilde L_t$,  then for every $f\in L^1\cap\s^2(X)$ we have
\[
L_t(f)=\tilde L_t(f),\qquad a.e.\ t\in[0,1].
\]
In this sense, for a given solution of the continuity equation the family $\{L_t\}_{t\in[0,1]}$ is essentially uniquely defined.

Our main result is that for curves of bounded compression, the continuity equation characterizes 2-absolute continuity.
\begin{theorem}\label{thm:main}
Let $(\mu_t)\subset \probt X$ be a $W_2$-continuous curve with bounded compression. Then the following are equivalent.
\begin{itemize}
\item[i)]  $(\mu_t)$ is  2-absolutely continuous w.r.t. $W_2$.
\item[ii)] There is a family of maps $\{L_t\}_{t\in[0,1]}$ from $\s^2(X)$ to $\R$ such that $(\mu_t)$ solves the continuity equation \eqref{eq:basecont}.
\end{itemize}
If these hold, we have
\[
N_t=|\dot\mu_t|,\qquad a.e.\ t\in[0,1].
\]
\end{theorem}
\begin{proof}$\ $\\
\noindent{$\mathbf{(i)\Rightarrow (ii)}$} Let $\ppi$ be a lifting of $(\mu_t)$ and notice that $\ppi$ is a test plan. Hence for $f\in L^1\cap\s^2(X)$ we have
\begin{equation}
\label{eq:easy}
\begin{split}
\left|\int f\,\d\mu_s-\int f\,\d\mu_t\right|&\leq \int |f(\gamma_s)-f(\gamma_t)|\,\d\ppi(\gamma)\leq\iint_t^s|Df|(\gamma_r)|\dot\gamma_r|\,\d r\,\d\ppi(\gamma)\\
&\leq \sqrt{\int_t^s\int|Df|^2\,\d\mu_r\,\d r}\sqrt{\int_t^s\int|\dot\gamma_r|^2\,\d \ppi(\gamma)\,\d r}.
\end{split}
\end{equation}
Taking into account that $\int|Df|^2\,\d\mu_r\leq C\int|Df|^2\,\d\mm$ for every $t\in[0,1]$, this shows that $t\mapsto \int f\,\d\mu_t$ is absolutely continuous.

Define $L:\s^2(X)\to L^1(0,1)$ by $L(f)(t):=\partial_t\int f\,\d\mu_t$ and notice that the bound \eqref{eq:easy} gives
\[
\left|\int_t^sL(f)(r)\,\d r\right|\leq  \sqrt{\int_t^s\int|Df|^2\,\d\mu_r\,\d r}\sqrt{\int_t^s G_r^2\,\d r},\qquad\forall t,s\in[0,1],\ t<s,
\]
for $G_t:=|\dot\mu_t|=\sqrt{\int|\dot\gamma_r|^2\,\d \ppi(\gamma)}\in L^2(0,1)$. Hence we can apply Lemma \ref{le:local} and deduce from \eqref{eq:bound} that for every $f\in\s^2(X)$ we have
\[
L_t(f)-\frac12\|f\|_{\mu_t}^2\leq \frac12|\dot\mu_t|^2,\qquad a.e.\ t\in[0,1].
\]
By the definition \eqref{eq:normt}, this latter bound is equivalent to  $N_t\leq|\dot\mu_t|$ for a.e. $t\in[0,1]$.

\noindent{$\mathbf{(ii)\Rightarrow (i)}$} To get the result it is sufficient to prove that
\[
W_2^2(\mu_t,\mu_s)\leq |s-t|\int_t^sN^2_r\,\d r,\qquad\forall t,s\in[0,1],\ t<s.
\]
We shall prove this bound for $t=0$ and $s=1$ only, the general case following by a simple rescaling argument. Recalling that
\[
\frac12W_2^2(\mu_0,\mu_1)=\sup_{\psi}\int \psi\,\d\mu_0+\int \psi^c\,\d\mu_1=\sup_{\varphi}\int Q_1\varphi\,\d\mu_1-\int \varphi\,\d\mu_0,
\]
the sup being taken among all Lipschitz and bounded $\psi,\varphi$, to get the claim it is sufficient to prove that
\begin{equation}
\label{eq:kuwada}
\int Q_1\varphi\,\d\mu_1-\int \varphi\,\d\mu_0\leq \frac12\int_0^1 N_t^2 \,\d t,
\end{equation}
for any Lipschitz and bounded $\varphi:X\to\R$. Fix such $\varphi$ and notice that
\begin{equation}
\label{eq:part0}
\begin{split}
\int Q_{1}\varphi\,\d\mu_{1}-\int \varphi\,\d\mu_{0}=\lim_{n\to\infty}\bigg\{&\sum_{i=0}^{n-1}\int (Q_{\frac{i+1}{n}}\varphi-Q_{\frac in}\varphi)\,\d\mu_{\frac{i+1}{n}}+\int Q_{\frac{i}n}\varphi\,\d(\mu_{\frac{i+1}{n}}-\mu_{\frac in})\bigg\}.
\end{split}
\end{equation}
Recalling point $(ii)$ of Proposition \ref{prop:HL} we have
\[
\begin{split}
\sum_{i=0}^{n-1}\int (Q_{\frac{i+1}{n}}\varphi-Q_{\frac in}\varphi)\,\d\mu_{\frac in}&\leq\sum_{i=0}^{n-1}\int\int_{\frac in}^{\frac{i+1}n}-\frac{\lip(Q_t\varphi)^2}{2}\,\d t\,\d\mu_{\frac in}=\int_{X\times[0,1]}-\frac{\lip(Q_t\varphi)^2(x)}2\,\d\mmu_n(x,t).
\end{split}
\]
where $\mmu_n:=\sum_{i=0}^{n-1}\mu_{\frac in}\times\mathcal L^1\restr{[\frac in,\frac{i+1}n]}$. The continuity of $(\mu_t)$ easily yields that $(\mmu_n)$ converges to $\mmu:=\d\mu_t(x)\otimes \d t$ in duality with $C_b(X\times [0,1])$. Furthermore, the assumption $\mu_t\leq C\mm$ for every $t\in[0,1]$ yields $\mmu_n\leq C\mm\times\mathcal L^1$ for every $n\in\N$ and thus by the Dunfort-Pettis theorem (see for instance Theorem 4.7.20 in \cite{Bogachev07}) we deduce that $(\mmu_n)$ converges to $\mmu\in\prob{X\times[0,1]}$, $\d\mmu:=\d\mu_t\otimes\d t$,  in duality with $L^\infty(X\times[0,1])$. Being $(t,x)\mapsto \frac{\lip(Q_t\varphi)^2(x)}2$ bounded (point $(i)$ of Proposition \ref{prop:HL}), we deduce that
\begin{equation}
\label{eq:part1}
\lims_{n\to\infty}\sum_{i=0}^{n-1}\int (Q_{\frac{i+1}{n}}\varphi-Q_{\frac in}\varphi)\,\d\mu_{\frac in}\leq \iint_0^1-\frac{\lip(Q_t\varphi)^2(x)}2\,\d\mu_t\,\d t.
\end{equation}
On the other hand we have
\[
\begin{split}
\sum_{i=0}^{n-1}\int Q_{\frac{i}n}\varphi\,\d(\mu_{\frac{i+1}n}-\mu_{\frac in})&=\sum_{i=0}^{n-1}\int_{\frac in}^{\frac{i+1}n}L_s(Q_{\frac{i}n}\varphi)\,\d s\\
&\leq \sum_{i=0}^{n-1}\frac12\int_{\frac in}^{\frac{i+1}n}N_s^2\,\d s+\sum_{i=0}^{n-1}\int_{\frac in}^{\frac{i+1}n}\int\frac{|DQ_{\frac{i}n} \varphi |^2}2\,\d\mu_s\,\d s\\
&\leq \frac12\int_0^1N_t^2\,\d t+\int_{X\times[0,1]} f_n(t,x)\,\d\mmu,
\end{split}
\]
where $f_n(t,x):=\frac{\lip(Q_{\frac{i}n}\varphi)^2(x)}{2}$ for $t\in[\frac in,\frac{i+1}n)$ and $\d\mmu(t,x):=\d\mu_t(x)\otimes \d t$. Recall that by points $(i),(iii)$ of Proposition \ref{prop:HL}  we have that the $f_n$'s are equibounded and satisfy $\lims_n f_n(t,x)\leq f(t,x):=\frac{\lip(Q_{t}\varphi)^2(x)}{2}$, thus Fatou's lemma gives
\[
\lims_{n\to\infty}\int_{X\times[0,1]} f_n(t,x)\,\d\mmu\leq \int_{X\times[0,1]} f(t,x)\,\d\mmu,
\]
and therefore
\begin{equation}
\label{eq:part2}
\lims_{n\to\infty}\sum_{i=0}^{n-1}\int Q_{\frac{i+1}n}\varphi\,\d(\mu_{\frac{i+1}n}-\mu_{\frac in})\leq \frac12\int_0^1N_t^2\,\d t+\iint_0^1\frac{\lip(Q_t\varphi)^2}{2}\,\d \mu_t\,\d t
\end{equation}
The bounds \eqref{eq:part1} and \eqref{eq:part2} together with \eqref{eq:part0} give  \eqref{eq:kuwada} and the thesis.
\end{proof}

If we know that $\s^2(X)$ is separable, the result is slightly stronger, as a better description of the operators $\{L_t\}$ is possible, as shown by the following statement.
\begin{proposition}\label{cor:separ}
Let $(\mu_t)\subset \probt X$ be a 2-absolutely continuous curve w.r.t. $W_2$ of bounded compression. Assume furthermore that $\s^2(X)$ is separable. Then there exists a $\mathcal L^1$-negligible set $\mathcal N\subset[0,1]$ and, for every $t\in[0,1]\setminus \mathcal N$, a linear map $L_t:\s^2(X)\to\R$ such that:
\begin{itemize}
\item[i)] every $t\in[0,1]\setminus\mathcal N$ is a Lebesgue point of $s\mapsto|\dot\mu_s|^2$, the metric speed  $|\dot\mu_s|$ exists at $s=t$ and we have $|\dot\mu_t|=\|L_t\|^*_{\mu_t}$,\\
\item[ii)] for every $f\in L^1\cap\s^2(X)$ the map $t\mapsto\int f\,\d\mu_t$ is absolutely continuous, differentiable at every $t\in[0,1]\setminus\mathcal L$ and its derivative is given by
\[
\frac{\d}{\d t}\int f\,\d\mu_t=L_t(f),\qquad\forall t\in[0,1]\setminus\mathcal N.
\]
\end{itemize}
\end{proposition}
\begin{proof}
Let $\{f_n\}_{n\in\N}\subset \s^2(X)$ be a countable dense set and $\mathcal N\subset[0,1]$ the set of $t$'s such that either the metric speed $|\dot\mu_t|$ does not exist, or $t$ is not a Lebesgue point of $s\mapsto|\dot\mu_s|^2$ or for some $n\in\N$ the map $s\mapsto\int f\,\d\mu_s$ is not differentiable at $t$. Then by Theorem \ref{thm:main} we know that $\mathcal N$ is negligible.

For $n\in\N$ and $t\in[0,1]\setminus\mathcal N$, inequality \eqref{eq:easy} gives, after a division for $|s-t|$ and a limit $s\to t$, the bound
\[
\left|\frac{\d}{\d t}\int f_n\,\d\mu_t\right|\leq |\dot\mu_t|\sqrt{\int |Df_n|^2\,\d\mu_t}\leq C |\dot\mu_t|\sqrt{\int |Df_n|^2\,\mm}.
\]
This means that the map $\s^2(X)\ni f_n\mapsto \frac{\d}{\d t}\int f_n\,\d\mu_t$ can be uniquely extended to a linear operator $L_t$ from $\s^2(X)$ to $\R$ which satisfies $\|L_t\|^*_{\mu_t}\leq |\dot\mu_t|$.

For $f\in L^1\cap\s^2(X)$ denote by $I_f:[0,1]\to\R$ the function given by $I_f(t):=\int f\,\d(\mu_t-\mu_0)$. Then the map $f\mapsto I_f$ is clearly linear and satisfies
\[
\begin{split}
|I_f(t)|\leq\int |f(\gamma_t)-f(\gamma_0)|\,\d\ppi(\gamma)&\leq \iint_0^t|Df|(\gamma_s)|\dot\gamma_s|\,\d s\,\d\ppi(\gamma)\leq\sqrt C\|f\|_{\s^2}\sqrt{\iint_0^1|\dot\gamma_t|^2\,\d\ppi(\gamma)}.
\end{split}
\]
Hence given that we have $I_{f_n}(t)=\int_0^tL_t(f_n)\,\d t$ for every $n\in\N$ and that $\{f_n\}$ is dense in $L^1\cap\s^2(X)$ w.r.t. the (semi)distance of $\s^2(X)$, from the bound $N_t\leq |\dot\mu_t|$ for $\mathcal L^1$-a.e. $t$, we deduce that $I_f(t)=\int_0^t L_t(f)\,\d t$ for every $f\in L^1\cap \s^2(X)$ and every $t\in[0,1]$.

Along the same lines we have that
\[
\begin{split}
|I_f(s)-I_f(t)|&\leq \int |f(\gamma_s)-f(\gamma_t)|\,\d\ppi(\gamma)\leq\iint_t^s |Df|(\gamma_r)|\dot\gamma_r|\,\d r\,\d\ppi(\gamma)\leq\\
&\leq \sqrt{C|s-t|\iint_t^s|\dot\gamma_s|^2\,\d r\,\d\ppi(\gamma)}\|f\|_{\s^2}= \sqrt{C|s-t|\int_t^s|\dot\mu_r|^2\,\d r}\|f\|_{\s^2}
\end{split}
\]
and therefore for every $t\in[0,1]$ Lebesgue point of $s\mapsto|\dot\mu_s|^2$ and such that $|\dot\mu_t|$ exists we have
\[
\lims_{s\to t}\left|\frac{I_f(s)-I_f(t)}{s-t}\right|\leq \sqrt C\|f\|_{\s^2}|\dot\mu_t|.
\]
Taking into account that, by construction, we have $\lim_{s\to t}\frac{I_{f_n}(s)-I_{f_n}(t)}{s-t}=L_t(f_n)$ for every $t\in[0,1]\setminus\mathcal N$ and the density of $\{f_n\}$, we deduce that $\lim_{s\to t}\frac{I_{f}(s)-I_{f}(t)}{s-t}=L_t(f)$ for every $f\in L^1\cap \s^2(X)$ and  $t\in[0,1]\setminus\mathcal N$ .

It remains to prove that $\|L_t\|^*_{\mu_t}=|\dot\mu_t|$ for $t\in[0,1]\setminus\mathcal N$. From Theorem \ref{thm:main} we know that $N_t=|\dot\mu_t|$ for $\mathcal L^1$-a.e. $t$ and to conclude use the separability of $\s^2(X)$ to get
\[
\frac12(\|L_t\|^*_{\mu_t})^2=\sup_{n\in\N}L_t(f_n)-\frac12\|f_n\|^2_{\mu_t}=\esssup_{f\in\s^2(X)}L_t(f)-\frac12\|f\|^2_{\mu_t}=\frac12 N_t^2,\qquad a.e.\ t,
\]
so that up to enlarging $\mathcal N$ we get the thesis.
\end{proof}

\subsection{Some consequences in terms of differential calculus}
As discussed in \cite{Otto01}, see also \cite{AmbrosioGigliSavare08}, the continuity equation plays a key role in developing a first order calculus on the space $(\probt{\R^d},W_2)$. In this section, we show that the continuity equation plays a similar role on metric measure spaces, where no smooth structure is a priori given. The only technical difference one needs to pay attention to is the fact that only curves with bounded compression should be taken into account.

\bigskip

We start with the Benamou-Brenier formula. Recall that on $\R^d$, and more generally on Riemannian/Finslerian manifolds, we have the identity
\begin{equation}
\label{eq:BB}
W_2^2(\mu_0,\mu_1)=\inf\int_0^1\int |v_t|^2\,\d\mu_t\,\d t,
\end{equation}
where the $\inf$ is taken among all 2-a.c. curves $(\mu_t)$ joining $\mu_0 $ to $\mu_1$ and the $v_t$'s are such that the continuity equation
\begin{equation}
\label{eq:contsmooth}
\partial_t\mu_t+\nabla\cdot(v_t\mu_t)=0,
\end{equation}
holds. We want to investigate the validity of this formula in the metric-measure context. To this aim, notice that formula \eqref{eq:BB} expresses the fact that the distance $W_2$ can be realized as $\inf$ of length of curves, where this length is measured in an appropriate way. Hence there is little hope to get an analogous of this formula on $(X,\sfd,\mm)$ unless we require in advance that $(X,\sfd)$ is a length space. Furthermore, given that in the non-smooth case we are confined to work with curves with bounded compression, we need to enforce a length structure compatible with the measure $\mm$, thus we are led to the following definition:
\begin{definition}[Measured-length spaces]
We say that $(X,\sfd,\mm)$ is measured-length provided for any $\mu_0,\mu_1\in\probt X$ with bounded support and satisfying $\mu_0,\mu_1\leq C\mm$ for some $C>0$ the distance $W_2(\mu_0,\mu_1)$ can be realized as $\inf$ of length of absolutely continuous curves $(\mu_t)$ with bounded compression connecting $\mu_0$ to $\mu_1$.
\end{definition}
On measured-length spaces we then have a natural analog of formula \eqref{eq:BB}, which is in fact a direct consequence of Theorem \ref{thm:main}:
\begin{proposition}[Benamou-Brenier formula on metric measure spaces]
Let $(X,\sfd,\mm)$ be a measured-length space and $\mu_0,\mu_1\in\probt X$ with bounded support and satisfying $\mu_0,\mu_1\leq C\mm$ for some $C>0$.

Then we have
\[
W_2^2(\mu_0,\mu_1)=\inf\int_0^1(\|L_t\|^*_{\mu_t})^2\,\d t,
\]
the $\inf$ being taken among all 2-absolutely continuous curves $(\mu_t)$ with bounded compression joining $\mu_0$ to $\mu_1$ and the operators $(L_t)$ are those associated to the curve via Theorem \ref{thm:main}.
\end{proposition}
\begin{proof}
By Theorem \ref{thm:main} we know that for a 2-absolutely continuous curve $(\mu_t)$ with bounded compression we have $|\dot\mu_t|=\|L_t\|_{\mu_t}^*$ for a.e. $t\in[0,1]$, the operators $\{L_t\}$ being those associated to the curve via Theorem \ref{thm:main} itself. The conclusion then follows directly from the definition of measured-length space.
\end{proof}
We now discuss the formula for the derivative of $t\mapsto\frac12W_2^2(\mu_t,\nu)$, where $(\mu_t)$ is a 2-a.c. curve with bounded compression. Recall that on the Euclidean setting we have
\[
\frac\d{\d t}\frac12W_2^2(\mu_t,\nu)=\int \nabla\varphi_t\cdot v_t\,\d\mu_t,\qquad a.e.\ t,
\]
where $\varphi_t$ is a Kantorovich potential from $\mu_t$ to $\nu$ for every $t\in[0,1]$ and the vector fields $(v_t)$ are such that the continuity equation \eqref{eq:contsmooth} holds. Due to our interpretation of the continuity equation in the metric measure setting, we are therefore lead to guess that in the metric-measure setting we have
\begin{equation}
\label{eq:derw}
\frac\d{\d t}\frac12W_2^2(\mu_t,\nu)=L_t(\varphi_t),\qquad a.e.\ t.
\end{equation}
As we shall see in a moment (Proposition \ref{prop:derw2}) this is actually the case in quite high generality, but before coming to the proof, we need to spend few words on how to interpret the right hand side of \eqref{eq:derw} because in general we don't have $\varphi_t\in\s^2(X)$ so that a priory $\varphi_t$ is outside the domain of definition of $L_t$. This can in fact be easily fixed by considering  $\varphi_t$  as element of ${\rm CoTan}_\mu(X)$, as defined in Section \ref{se:def}. This is the scope of the following lemma.
\begin{lemma}\label{le:1}
Let $(X,\sfd,\mm)$ be a m.m.s. such that $\mm$ gives finite mass to bounded sets and  $\varphi$  a $c$-concave function such that $\partial^c\varphi(x)\cap B\neq\emptyset$ for every $x\in X$ and some bounded set $B\subset X$.  Define  $\varphi_n:=\min\{n,\varphi\}$.

Then  $\varphi_n\in \s^2(X)$ and $(\varphi_n)$ is a Cauchy sequence w.r.t. the seminorm $\|\cdot\|_\mu$ for every $\mu\in\probt X$ such that $\mu\leq C\mm$ for some $C>0$.
 \end{lemma}
 \begin{proof}
 We first claim that $\sup_B\varphi^c<\infty$. Indeed, if not there is a sequence $(y_n)\subset B$ such that $\varphi^c(y_n)>n$ for every $n\in\N$. Hence for every $x\in X$ we would have
 \[
 \varphi(x)\leq \inf_{n\in\N}\sfd^2(x,y_n)-\varphi^c(y_n)\leq\inf_{n\in\N}\frac12\big(\sfd(x,B)+{\rm diam}(B)\big)^2-n=-\infty,
 \]
 contradicting the definition of $c$-concavity. Using the assumption we have
 \begin{equation}
\label{eq:boundvarphi}
 \varphi(x)=\inf_{y\in B}\frac{\sfd^2(x,y)}2-\varphi^c(y)\geq \frac{\sfd^2(x,B)}2-\sup_B\varphi^c.
\end{equation}
This proves that $\varphi$ is bounded from below and that it has bounded  sublevels. Hence the truncated functions $\varphi_n$ are constant outside a bounded set. Now let $x,x'\in X$ and $y\in\partial^c\varphi(x)\cap B$. Then we have
 \[
 \varphi(x)-\varphi(x')\leq\frac{\sfd^2(x,y)-\sfd^2(x'y)}2\leq\sfd(x,x')\left(\frac{\sfd(x,B)+\sfd(x',B)}2+{\rm diam}(B)\right).
 \]
 Inverting the roles of $x,x'$ we deduce that $\varphi$ is Lipschitz on bounded sets and the pointwise estimate
 \begin{equation}
\label{eq:estlip}
\lip(\varphi)(x)\leq\sfd(x,B)+{\rm diam}(B),\qquad\forall x\in X.
\end{equation}
It follows that the $\varphi_n$'s are Lipschitz and, using the fact that $\mm$ gives finite mass to bounded sets, that $\varphi_n\in\s^2(X)$ for every $n\in\N$.

To conclude, notice that the bound \eqref{eq:estlip} ensures that for any $\mu\in\probt X$ we have $\lip(\varphi)\in L^2(\mu)$. Thus for $\mu$ such that $\mu\leq C\mm$ for some $C>0$ we have $|D\varphi|\leq \lip(\varphi)$ $\mu$-a.e. and thus $|D\varphi|\in L^2(\mu)$ as well. Now observe  that
\[
\begin{split}
\|\varphi_m-\varphi_n\|_\mu^2=\int_{\{\varphi_m\neq\varphi_n\}}|D\varphi|^2\,\d\mu,
\end{split}
\]
and that the right hand side goes to 0 as $n,m\to\infty$, because by \eqref{eq:boundvarphi} we know that  $\cup_n\{\varphi=\varphi_n\}=X$.
 \end{proof}
Thanks to this lemma we can, and will, associate { to} the Kantorovich potential $\varphi$ an element of ${\rm CoTan}_\mu(X)$: it is the limit of the equivalence classes of the truncated functions $\varphi_n$.

Recall that for $\mu,\nu\in\probt X$, there always exists a Kantorovich potential $\varphi$ from $\mu$ to $\nu$ such that
\begin{equation}
\label{eq:4}
\varphi(x)=\inf_{y\in\supp(\nu)}\frac{\sfd^2(x,y)}2-\varphi^c(y),\qquad\forall x\in X,
\end{equation}
hence if $\nu$ has bounded support, a potential satisfying the assumption of Lemma \ref{le:1} above can always be found.

We can now state and prove the following result about the derivative of $W_2^2(\cdot,\nu)$. It is worth noticing that formula \eqref{eq:derw2} below holds even for spaces which are not length spaces.
\begin{proposition}[Derivative of ${W}^2_2(\cdot, \nu)$]\label{prop:derw2}
Let $(\mu_t)\subset\probt X$ be a 2-a.c. curve with bounded compression, $\nu\in\probt X$ with bounded support and notice that $t\mapsto\frac12 W_2^2(\mu_t,\nu)$ is absolutely continuous. Assume that $\s^2(X)$ is separable and that $\mm$ gives finite mass to bounded sets. Then for a.e. $t\in[0,1]$ the formula
\begin{equation}
\label{eq:derw2}
\frac{\d}{\d t}\frac12 W_2^2(\mu_t,\nu)=L_t(\varphi_t),
\end{equation}
holds, where $\varphi_t$ is any Kantorovich potential from $\mu_t$ to $\nu$ fulfilling the assumptions of Lemma \ref{le:1}.
\end{proposition}
\begin{proof}

Let $\mathcal N\subset [0,1]$ be the $\mathcal L^1$-negligible set given by Proposition \ref{cor:separ}.
We shall prove formula \eqref{eq:derw2} for every $t\in [0,1]\setminus \mathcal N$ such that $\frac12 W_2^2(\mu_\cdot,\nu)$ is differentiable at $t$. Fix such $t$, let $\varphi_t$ be as in the assumptions and notice that
\[
\begin{split}
\frac{1}{2}{W}^2_2(\mu_t, \nu)&=\int \varphi_t \,d\mu_t+ \int \varphi^c \,d\nu,\\
\frac{1}{2}{W}^2_2(\mu_s, \nu) &\geq \int \varphi_t \,d\mu_s+ \int \varphi^c \,d\nu,\qquad\forall s\in[0,1],
\end{split}
\]
and thus
\[
\frac{{W}^2_2(\mu_s, \nu)-{W}^2_2(\mu_t, \nu)}{2} \geq\int \varphi_t \d(\mu_s-\mu_t).
\]
Recall that $\max\{\varphi_t,0\}\in L^1(\mu)$ for every $\mu\in\probt X$ and that $\varphi_t\in L^1(\mu_t)$, so that the integral in the right hand side makes sense. Put $\varphi_{n,t}:=\min\{n,\max\{-n,\varphi_t\}\}$ so that by Lemma \ref{le:1} above we have $\varphi_{n,t}\in\s^2(X)$ for every $n\in\N$ and $\|\varphi_{n,t}-\varphi_{m,t}\|_\mu\to 0$ as $n,m\to\infty$. For every $n\in\N$ we know that $\frac{\d}{\d s}\int\varphi_{n,t}\,\d\mu_s\restr{s=t}=L_t(\varphi_{n,t})$ and by Lemma \ref{le:1} we know that $L_t(\varphi_{n,t})\to L_t(\varphi_t)$ as $n\to\infty$. To conclude it is sufficient to notice that for any lifting $\ppi$ of $(\mu_t)$ we have the bound
\[
\begin{split}
\left|\int(\varphi_{n,t} -\varphi_{m,t})\d\frac{\mu_s-\mu_t}{s-t} \right|&\leq \frac1{|s-t|}\int(\varphi_{n,t} -\varphi_{m,t})(\gamma_s)-(\varphi_{n,t} -\varphi_{m,t})(\gamma_t)\,\d\ppi(\gamma)\\
&\leq\frac1{|s-t|}\iint_t^s|D(\varphi_{n,t}-\varphi_{m,t})|(\gamma_r)|\dot\gamma_r|\,\d r\,\d\ppi(\gamma)\\
&\leq\sqrt{\dashint_t^s \|D(\varphi_{n,t}-\varphi_{m,t})\|^2_{\mu_r}\,\d r}\sqrt{\int\dashint_t^s|\dot\gamma_r|^2\,\d r\,\d\ppi(\gamma)},
\end{split}
\]
and  that the dominated convergence theorem ensures that $\dashint_t^s \|D(\varphi_{n,t}-\varphi_{m,t})\|^2_{\mu_r}\,\d r\to0$ as $n,m\to\infty$.
\end{proof}
\section{The continuity equation $\partial_t\mu_t+\nabla\cdot(\nabla\phi_t\mu_t)=0$}

\subsection{preliminaries: duality between differentials and gradients}

On Euclidean spaces it is often the case that the continuity equation \eqref{eq:contsmooth} can be written as
\begin{equation}
\label{eq:contgrad0}
\partial_t\mu_t+\nabla\cdot(\nabla\phi_t\mu_t)=0,
\end{equation}
for some functions $\phi_t$, so that the vector fields $v_t$ can be represented as gradient of functions. In some sense, the `optimal' velocity vector fields (i.e. those minimizing the $L^2(\mu_t)$-norm) can always be thought of as gradients, as they always belong to the closure of the space of gradients of smooth functions w.r.t. the $L^2(\mu_t)$-norm (i.e. they belong to the - dual of the - cotangent space ${\rm CoTan}_{\mu_t}(\R^d)$), see \cite{AmbrosioGigliSavare08}. Yet, the process of taking completion in general destroys the property of being the gradient of a smooth/Sobolev functions, so that technically speaking general absolutely continuous curves solve \eqref{eq:contsmooth} and only in some cases one can write it as in \eqref{eq:contgrad0}.

It is then the scope of this part of the paper to investigate how one can give a meaning to \eqref{eq:contgrad0} in the non-smooth setting and which sort of information on the curve we can obtain from such `PDE'. According to our interpretation of the continuity equation given in Theorem \ref{thm:main}, the problem reduces to understand in what sense we can write $L_t(f)=\int Df(\nabla\phi_t)\,\d\mu_t$, and thus ultimately to give a meaning to `the differential of a function applied to the gradient of another function'. This has been the scope of \cite{Gigli12}, we recall here the main definitions and properties.
\begin{definition}[The objects $D^\pm f(\nabla g)$]
Let $(X,\sfd,\mm)$ be a m.m.s. and $f,g\in\s^2(X)$.

The functions $D^\pm f(\nabla g) : X \to \mathbb{R}$ are $\mm$-a.e. well defined by
\[
\begin{split}
D^+f(\nabla g) &:= \lim_{\eps\downarrow0}\frac{|D(g+\eps f)|^2-|Dg|^2}{2\eps },\\
D^-f(\nabla g) &:= \lim_{\eps\uparrow0}\frac{|D(g+\eps f)|^2-|Dg|^2}{2\eps }.
\end{split}
\]
\end{definition}
It is immediate to check that for $\eps_1<\eps_2$ we have
\[
\frac{|D(g+\eps_1 f)|^2-|Dg|^2}{2\eps_1 }\leq \frac{|D(g+\eps_2 f)|^2-|Dg|^2}{2\eps_2 },\qquad\mm-a.e.,
\]
so that the limits above can be replaced by $\inf_{\eps>0}$ and $\sup_{\eps<0}$ respectively.

Heuristically, we should think to $D^+f(\nabla g)$ (resp. $D^-f(\nabla g)$) as the maximal (resp. minimal) value of the differential of $f$ applied to all possible gradients of $g$, see \cite{Gigli12} for a discussion on this topic.

The basic algebraic calculus rules for $D^\pm f(\nabla g)$ are the following:
\begin{align}
\label{eq:normpm}
|D^\pm (f_1-f_2)(\nabla g)|&\leq |D(f_1-f_2)||Dg|,\\
\nonumber
D^-f(\nabla g)&\leq D^+f(\nabla g),\\
\label{eq:signpm}
D^+(-f)(\nabla g)&=D^+f(\nabla (-g))=-D^-f(\nabla g),\\
\label{eq:squarepm}
D^\pm g(\nabla g)&=|Dg|^2.
\end{align}
We also have natural chain rules: given $\varphi:\R\to\R$ Lipschitz we have
\begin{equation}
\label{eq:chains}
\begin{split}
D^\pm (\varphi\circ f)(\nabla g)&=\varphi'\circ f D^{\pm{\rm sign}\varphi'\circ f}f(\nabla g),\\
D^\pm f (\nabla \varphi\circ g)&=\varphi'\circ g D^{\pm{\rm sign}\varphi'\circ g}f(\nabla g),
\end{split}
\end{equation}
where $\varphi'\circ f$ (resp. $\varphi'\circ g$) are defined arbitrarily at those $x$'s such that $\varphi$ is not differentiable at $f(x)$ (resp. $g(x)$). In particular, $ D^\pm f (\nabla(\alpha g))=\alpha D^{\pm}f(\nabla g)$ for $\alpha>0$.

Notice that as a consequence of the above we have that for given $g\in\s^2(X)$ the map $\s^2(X)\ni f\mapsto D^+f(\nabla g)$ is $\mm$-a.e. convex, i.e.
\begin{equation}
\label{eq:dpmconv}
D^+((1-\lambda)f_1+\lambda f_2)(\nabla g)\leq(1-\lambda) D^+f_1(\nabla g)+\lambda D^+f_2(\nabla g),\qquad\mm-a.e.,
\end{equation}
for every $f_1,f_2\in\s^2(X)$, and $\lambda\in[0,1]$. Similarly, $f\mapsto D^-f(\nabla g)$ is $\mm$-a.e. concave.

Furthermore, it is easy to see that for $g\in\s^2(X)$ and $\ppi$ test plan we have
\begin{equation}
\label{eq:allplans}
\lims_{t\downarrow0}\int\frac{g(\gamma_t)-g(\gamma_0)}{t}\,\d\ppi(\gamma)\leq \frac12\int|Dg|^2(\gamma_0)\,\d\ppi(\gamma)+\frac12\lims_{t\downarrow0}\frac1t\iint_0^t|\dot\gamma_s|^2\,\d s\,\ppi(\gamma).
\end{equation}
We are therefore lead to the following definition:
\begin{definition}[Plans representing gradients]
Let $(X,\sfd,\mm)$ be a m.m.s.  $g\in \s^2(X)$ and $\ppi$ a test plan.  We say that $\ppi$ represents the gradient of $g$ provided it is a test plan and we have
\begin{equation}
\label{eq:defplrepr}
\limi_{t\downarrow0}\int\frac{g(\gamma_t)-g(\gamma_0)}{t}\,\d\ppi(\gamma)\geq \frac12\int|Dg|^2(\gamma_0)\,\d\ppi(\gamma)+\frac12\lims_{t\downarrow0}\frac1t\iint_0^t|\dot\gamma_s|^2\,\d s\,\ppi(\gamma).
\end{equation}
\end{definition}
It is worth noticing that plans representing gradients exist in high generality (see \cite{Gigli12}).

Differentiation along plans representing gradients is tightly linked to the object $D^\pm f(\nabla g)$ defined above: this is the content of the following simple but crucial theorem proved in \cite{Gigli12} as a generalization of a result originally appeared in \cite{AmbrosioGigliSavare11-2}.
\begin{theorem}[Horizontal and vertical derivatives]
Let $(X,\sfd,\mm)$ be a m.m.s., $f,g\in \s^2(X)$ and $\ppi$ a plan representing the gradient of $g$.

Then
\begin{equation}
\label{eq:horver}
\begin{split}
\int D^- f(\nabla g)\,\d(\e_0)_\sharp\ppi &\leq\limi_{t\downarrow 0}\int \frac{f(\gamma_t) -f(\gamma_0)}{t} \,\d\ppi(\gamma)\\
&\leq \lims_{t\downarrow 0}\int \frac{f(\gamma_t) -f(\gamma_0)}{t} \,\d\ppi(\gamma) \leq \int D^+ f(\nabla g)\,\d(\e_0)_\sharp\ppi.
\end{split}
\end{equation}
\end{theorem}
\begin{proof}
Write inequality \eqref{eq:allplans} for the function $g+\eps f$ and subtract inequality \eqref{eq:defplrepr} to get
\[
\lims_{t\downarrow 0}\eps \int \frac{f(\gamma_t) -f(\gamma_0)}{t}\,\d \ppi \leq \frac{1}{2} \int |D(g+\eps f)|^2 - |Dg|^2\, \d(\e_0)_\sharp \ppi
\]
Divide by $\eps >0$ (resp. $\eps <0$), let $\eps\downarrow0$ (resp. $\eps\uparrow 0$)  and use the dominate convergence theorem to conclude.
\end{proof}

\subsection{The result}
We are now ready to define what it is a solution of the continuity equation \eqref{eq:contgrad0} in a metric measure context.

\begin{definition}[Solutions of $\partial_t\mu_t+\nabla \cdot(\nabla\phi_t\mu_t)=0$]\label{def:contgrad}
Let $(\mu_t)\subset \probt X$ be a $W_2$-continuous  curve with bounded compression and $\{\phi_t\}_{t\in[0,1]}\subset \s^2(X)$ a given family. We say that $(\mu_t)$ solves the continuity equation
\begin{equation}
\label{eq:contgrad}
 \partial_t \mu_t +\nabla \cdot (\nabla \phi_t \mu_t)=0,
\end{equation}
provided:
\begin{itemize}
\item[i)] for every $f\in\s^2(X)$ the maps $(t,x) \mapsto D^\pm f(\nabla\phi_t)(x)$ are $\mathcal L\times \mm$ measurable and the map $\tilde N:[0,1]\to[0,\infty]$ given by
\begin{equation}
\label{eq:tilden}
\frac12\tilde N^2_t:=\esssup_{f\in\s^2(X)}\int D^+f(\nabla \phi_t)\,\d\mu_t-\frac12\|f\|^2_{\mu_t}
\end{equation}
is in $L^2(0,1)$ where $\esssup$ has the same meaning as we mentioned before.
\item[ii)] for every $f\in L^1\cap \s^2(X)$ the map $t\mapsto\int f\,\d\mu_t$ is absolutely continuous and satisfies
\begin{equation}
\label{eq:asdef}
\int D^-f(\nabla \phi_t)\,\d\mu_t\leq\frac{\d}{\d t}\int f\,\d\mu_t\leq \int D^+f(\nabla\phi_t)\,\d\mu_t,\qquad a.e.\ t.
\end{equation}
\end{itemize}
\end{definition}

We then have the following result, analogous to the implication $(ii)\Rightarrow(i)$ of Theorem \ref{thm:main}. As already recalled, even in the smooth framework not all a.c. curves solve \eqref{eq:contgrad0}, so the other implication is in general false.
\begin{proposition}
Let $(X,\sfd,\mm)$ be a m.m.s. and $(\mu_t)\subset\probt X$ a continuous curve with bounded compression solving the continuity equation \eqref{eq:contgrad} for some given family $\{\phi_t\}_{t\in[0,1]}\subset \s^2(X)$.

Then $(\mu_t)$ is 2-a.c. and we have $|\dot\mu_t|\leq\tilde  N_t$ for a.e. $t\in[0,1]$.

If furthermore $\s^2(X)$ is separable, then $\tilde N_t=|\dot\mu_t|=\|\phi_t\|_{\mu_t}$ for a.e. $t$.
\end{proposition}
\begin{proof} We claim that for every $f\in\s^2(X)$ it holds
\begin{equation}
\label{eq:claimnorm}
\max\left\{\left|\int D^+f(\nabla\phi_t)\,\d\mu_t\right|,\left|\int D^-f(\nabla\phi_t)\,\d\mu_t\right|\right\}\leq \|f\|_{\mu_t} \tilde N_t,\qquad a.e.\ t.
\end{equation}
To this aim, fix a representative of $\tilde N$, a function $f\in\s^2(X)$ and notice that for every $\lambda\geq 0$, by definition of $\tilde N$ and the second of  \eqref{eq:chains} we have
\begin{equation}
\label{eq:3}
\lambda\int D^+f(\nabla\phi_t)\,\d\mu_t\leq \frac{\lambda^2}2\|f\|_{\mu_t}^2+\frac12\tilde N^2_t,
\end{equation}
for $\mathcal L^1$-a.e. $t$. Replacing $f$ with $-f$ and recalling that
\[
-\int D^+f(\nabla\phi_t)\,\d\mu_t=\int D^-(-f)(\nabla\phi_t)\,\d\mu_t\leq \int D^+(-f)(\nabla\phi_t)\,\d\mu_t,
\]
we deduce that \eqref{eq:3} holds for every $\lambda\in\R$. In particular, there is a $\mathcal L^1$-negligible set $\mathcal N\subset[0,1]$ such that for every $t\in[0,1]\setminus\mathcal N$  the inequality \eqref{eq:3} holds for every $\lambda\in\Q$. Given that all the terms in \eqref{eq:3} are continuous in $\lambda$, we deduce that \eqref{eq:3} holds for every $t\in[0,1]\setminus\mathcal N$ and every $\lambda\in\R$, which yields
\[
\left|\int D^+f(\nabla\phi_t)\,\d\mu_t\right|\leq \|f\|_{\mu_t} \tilde N_t,\qquad a.e.\ t.
\]
Arguing analogously with $D^-f(\nabla\phi_t)$ in place of $D^+f(\nabla\phi_t)$ we obtain \eqref{eq:claimnorm}.

Now define a linear operator $L:\s^2(X)\to L^1(0,1)$ as
\[
L(f)(t):=\frac\d{\d t}\int f\,\d\mu_t,
\]
and observe that for every $t,s\in[0,1]$, $t<s$, taking into account \eqref{eq:asdef} and \eqref{eq:claimnorm} we have
\[
\begin{split}
\int_t^s|L(f)(r)|\,\d r&\leq \int_t^s\max\left\{\left|\int D^+f(\nabla\phi_r)\,\d \mu_r\right|,\left|\int D^-f(\nabla\phi_r)\,\d \mu_r\right|\right\}\,\d r\\
&\leq \int_t^s\|f\|_{\mu_r}\tilde N_r\,\d r\leq \sqrt{\int_t^s\tilde N_r^2\,\d r} \sqrt{\int_t^s\|f\|^2_{\mu_r}\,\d r}.
\end{split}
\]
Hence we can apply first Lemma \ref{le:local} (with $G:=\tilde N$) and then Theorem \ref{thm:main} to deduce that $(\mu_t)$ is 2-a.c. with $|\dot\mu_t|\leq \tilde N_t$ for a.e. $t\in[0,1]$.

If $\s^2(X)$ is separable, then arguing exactly as in the proof of Proposition \ref{cor:separ} and using the convexity (resp. concavity) of $f\mapsto \int D^+f(\nabla\phi_t)\,\d\mu_t$ (resp. $f\mapsto \int D^-f(\nabla \phi_t)\,\d\mu_t$) expressed in \eqref{eq:dpmconv} we deduce the existence of a $\mathcal L^1$-negligible set $\mathcal N\subset[0,1]$ such that for $t\in[0,1]\setminus\mathcal N$ the conclusions $(i),(ii)$ of Proposition \ref{cor:separ} hold and furthermore
\[
\int D^-f(\nabla\phi_t)\,\d\mu_t\leq L_t(f)\leq \int D^+f(\nabla\phi_t)\,\d\mu_t,\qquad\forall f\in \s^2(X).
\]
Choosing $f:=\phi_t$ and recalling \eqref{eq:squarepm} we obtain
\[
\|\phi_t\|^2_{\mu_t}=L_t(\phi_t)\leq \|\phi_t\|_{\mu_t}\|L_t\|^*_{\mu_t},\qquad\forall t\in[0,1]\setminus\mathcal N,
\]
and hence $\|\phi_t\|_{\mu_t}\leq \|L_t\|^*_{\mu_t}=N_t=|\dot\mu_t|$ for a.e. $t$.  On the other hand, letting $(f_n)\subset \s^2(X)$ be a countable dense set, by \eqref{eq:normpm} we know that $\frac12\tilde N^2_t=\sup_nD^+f_n(\nabla\phi_t)-\frac12\|f_n\|_{\mu_t}^2$ for a.e. $t$ and thus $\tilde N_t\leq \|\phi_t\|_{\mu_t}$ for a.e. $t$.
\end{proof}
The continuity equation \eqref{eq:contgrad} has very general relations with the concept of `plans representing gradients', as shown by the following result:
\begin{theorem}\label{thm:contgrad}
Let $(\mu_t)\subset \probt X$ be a 2-a.c. curve with bounded compression, $(t,x)\mapsto\phi_t(x)$ a Borel map such that $\phi_t\in \s^2(X)$ for a.e. $t\in[0,1]$ and $\ppi$ a lifting of $(\mu_t)$.

Then the following are true.
\begin{itemize}
\item[i)] Assume that $({\rm restr}_t^{1})_\sharp\ppi$ represents the gradient of $(1-t)\phi_t$ for a.e. $t\in[0,1]$. Then $(\mu_t)$ solves the continuity equation \eqref{eq:contgrad}.
\item[ii)] Assume that $\s^2(X)$ is separable and that $(\mu_t)$ solves the continuity equation \eqref{eq:contgrad}. Then $({\rm restr}_t^{1})_\sharp\ppi$ represents the gradient of $(1-t)\phi_t$ for a.e. $t\in[0,1]$.
\end{itemize}
\end{theorem}
\begin{proof}$\ $\\
\noindent{\bf (i)} Let $\mathcal A\subset (0,1)$ be the set of $t$'s such that  $({\rm restr}_t^{1})_\sharp\ppi$ represents the gradient of $(1-t)\phi_t$, so that by assumption we know that $\mathcal L^1(\mathcal A)=1$. Pick $f\in\s^2(X)$ and recall that by Theorem \ref{thm:main} we know that
\begin{equation}
\label{eq:1}
\frac \d{\d t}\int f\,\d\mu_t=L_t(f),\qquad a.e. \ t\in[0,1].
\end{equation}
Fix $t\in \mathcal A$ such that \eqref{eq:1} holds and notice that
\[
\begin{split}
\frac \d{\d t}\int f\,\d\mu_t=\lim_{h\downarrow0}\int \frac{f(\gamma_{t+h})-f(\gamma_t)}{h}\,\d\ppi(\gamma)=\frac{1}{1-t}\lim_{h\downarrow0}\int \frac{f(\gamma_{h})-f(\gamma_0)}{h}\,\d\ppi_t(\gamma),
\end{split}
\]
so that recalling \eqref{eq:horver} and \eqref{eq:chains} we conclude.

\noindent{\bf (ii)} With exactly the same approximation procedure used in the proof of Proposition \ref{cor:separ}, we see that there exists a $\mathcal L^1$-negligible set $\mathcal N\subset[0,1]$ such that the thesis of Proposition \ref{cor:separ} is fulfilled and furthermore for every $t\in [0,1]\setminus\mathcal N$ we have
\begin{equation}
\label{eq:gradpoint}
\int D^{-} f(\nabla \phi_t)\,\d\mu_t \leq L_t(f)\leq  \int D^{+} f(\nabla \phi_t)\,\d\mu_t,\qquad \forall f\in \s^2(X).
\end{equation}
Fix $t\in[0,1]\setminus\mathcal N$ and observe that by point $(i)$ of Proposition \ref{cor:separ} we have that
\begin{equation}
\label{eq:normppi}
|\dot\mu_t|^2=\lim_{h\downarrow0}\frac1h\iint_t^{t+h}|\dot\gamma_s|^2\,\d\ppi(\gamma)=\frac1{(1-t)^2}\lim_{h\downarrow0}\frac1h\iint_0^h|\dot\gamma_s|^2\,\d\ppi_t(\gamma).
\end{equation}
Now pick $f:=\phi_t$ in \eqref{eq:gradpoint} and recall the identity $D^\pm f(\nabla f)=|Df|^2$ $\mm$-a.e. valid for every $f\in\s^2(X)$ to deduce
\[
\int|D\phi_t|^2\,\d\mu_t=L_t(\phi_t)=\lim_{h\downarrow0}\int \phi_t\,\d\frac{\mu_{t+h}-\mu_t}{h}=\frac1{1-t}\lim_{h\downarrow0}\int\frac{\phi_t(\gamma_h)-\phi_t(\gamma_0)}{h}\,\d\ppi_t(\gamma).
\]
This last identity, \eqref{eq:normppi} and the fact that $\|L_t\|_{\mu_t}^*=|\dot\mu_t|$ ensure  that $\ppi_t$ represents the gradient of $(1-t)\phi_t$, as claimed.
\end{proof}

\section{Two important examples}
We conclude the paper discussing two important examples of absolutely continuous curves on $\probt X$: the heat flow and the geodesics. These examples already appeared in the literature \cite{Gigli-Kuwada-Ohta10,AmbrosioGigliSavare11,Gigli13}, we report them here only to show that they are consistent with the concepts we introduced.

We start with the heat flow. Recall that the map {(called Cheeger energy in some of the bibliographical references)} $E:L^2(X,\mm)\to [0,\infty]$ given by
\[
E(f):=\left\{\begin{array}{ll}
\displaystyle{\frac12\int |Df|^2\,\d\mm},&\qquad\textrm{ if }f\in\s^2(X),\\
+\infty,&\qquad\textrm{ otherwise},
\end{array}
\right.
\]
is convex, lower semicontinuous and with dense domain. Being $L^2(X)$ an Hilbert space, we then know by the classical theory of gradient flows in Hilbert spaces (see e.g. \cite{AmbrosioGigliSavare08} and references therein) that for any $\rho\in L^2(X)$ there exists a unique continuous curve $[0,\infty)\ni t\mapsto \rho_t\in L^2(X)$ with $f_0=f$ which is locally absolutely continuous on $(0,\infty)$ and that satisfies
\[
\frac\d{\d t}\rho_t\in-\partial^-E(\rho_t),\qquad a.e. \ t>0.
\]
As in \cite{AmbrosioGigliSavare11}, we shall call any such gradient flow a \emph{heat flow}. It is immediate to check that defining $D(\Delta):=\{\rho:\partial^-E(\rho)\neq\emptyset\}$ and for  $\rho\in D(\Delta)$ the Laplacian $\Delta \rho$ as the opposite of the element of minimal norm in $\partial^-E(\rho)$, for any heat flow $(\rho_t)$ we have $\rho_t\in D(\Delta)$ for any $t>0$ and
\[
\frac{\d}{\d t}\rho_t=\Delta \rho_t,\qquad a.e. \ t>0,
\]
in accordance with the classical case. It is our aim now to check that, under reasonable assumptions, putting $\mu_t:=\rho_t\mm$, the curve $(\mu_t)$ solves
\[
\partial_t\mu_t+\nabla\cdot(\nabla(-\log\rho_t)\mu_t)=0.
\]
To this aim, recall that for any heat flow $(\rho_t)$ we have the weak maximum principle
\begin{equation}
\label{eq:maxpr}
\rho_0\leq C \textrm{ (resp. $\rho_0\geq c$) }\mm-a.e.\ \qquad\Rightarrow  \qquad\rho_t\leq C \textrm{ (resp. $\rho_t\geq c$) }\mm-a.e.\textrm{ for any $t>0$},
\end{equation}
where $c,C\in\R$ and the estimate
\begin{equation}
\label{eq:controlspeed}
\int_0^\infty \|\Delta\rho_t\|^2_{L^2(X)}\,\d t\leq E(\rho_t).
\end{equation}
Furthermore, if $\mm\in\prob X$ then $L^2(X,\mm)\subset L^1(X,\mm)$ and the mass preservation property holds:
\begin{equation}
\label{eq:mass}
\int\rho_t\,\d\mm=\int \rho_0\,\d\mm,\qquad\forall t>0.
\end{equation}
See \cite{AmbrosioGigliSavare11} for the simple proof of these facts.

We can now state our result concerning the heat flow as solution of the continuity equation.

\begin{proposition}[Heat flow]\label{prop:heat} Let $(X,\sfd,\mm)$ be a metric measure space with $\mm\in\probt X$ and $\rho_0$ a probability density such that $c\leq \rho_0\leq C$ $\mm$-a.e. for some $c,C>0$ (in particular $\rho_0\in L^2(X,\mm)$) and $E(\rho_0)<\infty$. Let $(\rho_t)$ be the heat flow starting from  from $\rho_0$.

Then the curve $[0,1]\ni t\mapsto\mu_t:=\rho_t\mm$ solves the continuity equation
\[
\partial_t\mu_t+\nabla\cdot(\nabla(-\log\rho_t)\mu_t)=0.
\]
\end{proposition}
\begin{proof} By \eqref{eq:mass} and \eqref{eq:maxpr} we know that $\rho_t$ is a probability density for every $t$ and \eqref{eq:maxpr} again and the assumptions $\mm\in\probt X$ and $\rho_0\leq C\mm$ ensure that $\mu_t\in\probt X$ for every $t\in[0,1]$ and that $(\mu_t)$ has bounded compression. The $W_2$-continuity of $(\mu_t)$ is a simple consequence of the $L^2$-continuity of $(\rho_t)$ and the bounds $\rho_t\leq C$, $\mm\in\probt X$. Also, recalling the chain rule \eqref{eq:chainbase} and the maximum principle \eqref{eq:maxpr} we have
\[
 \|\log\rho_t\|^2_{\mu_t}=\int|D(\log\rho_t)|^2\,\d\mu_t=\int \frac{|D\rho_t|^2}{\rho_t}\,\d\mm\leq \frac1c\int |D\rho_t|^2\,\d\mm,
\]
so that \eqref{eq:controlspeed} ensures that $\int_0^1\|\log\rho_t\|_{\mu_t}^2<\infty$, which directly yields that point $(i)$ of Definition \ref{def:contgrad} is fulfilled.

It remains to prove that for every $f\in L^1\cap\s^2(X)$ the map $t\mapsto\int f\,\d\mu_t$ is absolutely continuous and fulfills
\[
\int D^-f(\nabla(-\log\rho_t))\,\d\mu_t\leq \frac{\d}{\d t}\int f\,\d\mu_t\leq\int D^+f(\nabla(-\log\rho_t))\,\d\mu_t.
\]
Taking into account the chain rule \eqref{eq:chains} the above can be written as
\begin{equation}
\label{eq:toprove}
-\int D^+f(\nabla\rho_t)\,\d\mm\leq \frac{\d}{\d t}\int f\,\d\mu_t\leq-\int D^-f(\nabla\rho_t)\,\d\mm.
\end{equation}
Pick  $f\in L^2\cap\s^2(X)$ and notice that $t\mapsto\int f\,\d\mu_t=\int f\rho_t\,\d\mm$ is continuous on $[0,1]$ and locally absolutely continuous on $(0,1]$. The inequality
\[
\left|\frac\d{\d t}\int f\rho_t\,\d\mm\right|\leq \int|f||\Delta\rho_t|\,\d\mm\leq\frac12\|f\|^2_{L^2}+\frac12\|\Delta\rho_t\|^2_{L^2},
\]
the bound \eqref{eq:controlspeed} and the assumption $E(\rho_0)<\infty$ ensure that the derivative of $t\mapsto\int f\rho_t\,\d\mm$ is in $L^1(0,1)$, so that this function is absolutely continuous on $[0,1]$.

The fact that $-\frac\d{\d t}\rho_t\in-\partial^-E(\rho_t)$ for a.e. $t$ grants that for $\eps\in\R$ we have
\begin{equation}
\label{eq:heat}
\frac{\d}{\d t}\int \eps f\rho_t\,\d\mm=\int \eps f\frac{\d}{\d t}\rho_t\,\d\mm\leq E(\rho_t-\eps f)-E(f),\qquad a.e.\ t.
\end{equation}
Divide by $\eps>0$ and let $\eps\downarrow0$ to obtain
\[
\frac{\d}{\d t}\int f\rho_t\,\d\mm\leq\limi_{\eps\downarrow0}\frac{E(\rho_t-\eps f)-E(f)}{\eps}=\lim_{\eps\downarrow0}\int\frac{|D(\rho_t-\eps f)|^2-|D\rho_t|^2}{2\eps}\,\d\mm=-\int D^-f(\nabla\rho_t)\,\d\mm,
\]
which is the second inequality in \eqref{eq:toprove}. The first one is  obtained starting from \eqref{eq:heat}, dividing by $\eps<0$ and letting $\eps\uparrow0$.

The general case of $f\in L^1\cap\s^2(X)$ can now be obtained with a simple truncation argument, we omit the details.
\end{proof}

\bigskip

We now turn to the study of geodesics. In the smooth Euclidean/Riemannian framework a geodesic $(\mu_t)$ solves
\[
\partial_t\mu_t+\nabla\cdot(\nabla\phi_t\mu_t)=0,
\]
where $\phi_t:=Q_{1-t}(-\varphi^c)$ and $\varphi$ is a Kantorovich potential inducing $(\mu_t)$.

We want to show that the same holds on metric measure spaces, at least for geodesics with bounded compressions. This will be achieved as a simple consequence of Theorem \ref{thm:contgrad} and the following fact:
\begin{theorem}\label{thm:geod}
Let $(X,\sfd,\mm)$ be a m.m.s., $(\mu_t)\subset\probt X$ a geodesic with bounded compression and $\varphi\in\s^2(X)$ a Kantorovich potential from $\mu_0$ to $\mu_1$. Then:
\begin{itemize}
\item[i)] Any lifting $\ppi$ of $(\mu_t)$ represents the gradient of $-\varphi$.
\item[ii)] For any $t\in(0,1]$ the function $(1-t)Q_{1-t}(-\varphi^c)$ is a Kantorovich potential from $\mu_t$ to $\mu_1$.
\end{itemize}
\end{theorem}
\begin{proof}
Point $(i)$ of this theorem is a restatement of the metric Brenier theorem proved in \cite{AmbrosioGigliSavare11}, while point $(ii)$ is a general fact about optimal transport in metric spaces whose proof can be found in \cite{Villani09} or \cite{AmbrosioGigli11}.
\end{proof}

In stating the continuity equation for geodesics we shall make use of the fact that for $\mu,\nu\in\probt X$ with bounded support, there always exists a Kantorovich potential from $\mu$ to $\nu$ which is constant outside a bounded set: it is sufficient to pick any Kantorovich potential satisfying \eqref{eq:4} and proceed with a truncation argument. This procedure ensures that if $\mm$ gives finite mass to bounded sets, then these Kantorovich potentials are in $\s^2(X)$.

We then have the following result:
\begin{proposition}[Geodesics]\label{prop:geod}
Let $(X,\sfd,\mm)$ be a m.m.s. with $\mm$ giving finite mass to bounded sets, $(\mu_t)\subset\probt X$ a geodesic with bounded compression such that $\mu_0,\mu_1$ have bounded supports  and $\varphi$ a Kantorovich potential from $\mu_0$ to $\mu_1$ which is constant outside a bounded set.

Then
\[
\partial_t\mu_t+\nabla\cdot(\nabla\phi_t \mu_t)=0,
\]
where $\phi_t:=-Q_{1-t}(-\varphi^c)$ for every $t\in[0,1]$.
\end{proposition}
\begin{proof}
The assumption that $\varphi$ is constant outside a bounded set easily yields that $\varphi^c$ is Lipschitz and constant outside a bounded set and that for some $B\subset X$ bounded, $\phi_t$ is constant outside $B$ for any $t\in[0,1]$. Also, recalling point $(i)$ of Proposition \ref{prop:HL} we get that  the $\phi_t$'s are uniformly Lipschitz so that  the assumption that $\mm$ gives finite mass to bounded sets yields that $\sup_{t\in[0,1]}\|\phi_t\|_{\s^2}<\infty$. In particular, the function $\tilde N$ defined in \eqref{eq:tilden} is bounded and hence in $L^2(0,1)$, so that the statement makes sense.

Now let $\ppi$ be a lifting of $(\mu_t)$ and notice that $\ppi$ is a test plan so that $t\mapsto \int f\,\mu_t$ is absolutely continuous.  For $t\in [0,1)$ the plan $\ppi_t:=({\rm restr}_t^1)_\sharp\ppi$ is a lifting of $s\mapsto\mu_{t+s(1-t)}$. Thus by Theorem \ref{thm:geod} above we deduce that $\ppi_t$ represents the gradient of $(1-t)\phi_t$.

The conclusion follows by point $(i)$ of Theorem \ref{thm:contgrad}.
\end{proof}

In many circumstances, both heat flows and geodesics have regularity which go slightly beyond that of absolute continuity. Let us propose the following definition:
\begin{definition}[Weakly $C^1$ curves]
Let $(\mu_t)\subset\probt X$ be a 2-a.c. curve with bounded compression. We say that $(\mu_t)$ is weakly $C^1$ provided for any $f\in L^1\cap \s^2(X)$ the map $t\mapsto \int f\,\d\mu_t$ is $C^1$.
\end{definition}

In presence of weak $C^1$ regularity, the description of the operators $L_t$ in Theorem \ref{thm:main} can be simplified avoiding the use of the technical Lemma \ref{le:local}: it is sufficient to define $L_t:\s^2(X)\to\R$ by
\[
L_t(f):=\frac\d{\d t}\int f\,\d\mu_t,\qquad\forall f\in\s^2(X).
\]

Let us now discuss some cases where the heat flow and the geodesics are weakly $C^1$. We recall that $(X,\sfd,\mm)$ is said \emph{infinitesimally strictly convex} provided
\[
D^-f(\nabla g)=D^+f(\nabla g),\qquad\mm-a.e. \ \forall f,g\in\s^2(X).
\]
We then have the following regularity result:
\begin{proposition}[Weak $C^1$ regularity for the heat flow]
With the same assumptions of Proposition \ref{prop:heat}, assume furthermore that $(X,\sfd,\mm)$ is infinitesimally strictly convex.

Then $(\mu_t)$ is weakly $C^1$.
\end{proposition}
\begin{proof}
We have already seen in the proof of Proposition \ref{prop:heat} that for any $\rho\in D(\Delta)=D(\partial^-E)$ and $v\in -\partial^-E(\rho)\subset L^2(X,\mm)$ we have
\[
\int D^-f(\nabla \rho)\,\d\mm\leq \int fv\,\d\mm\leq \int D^+f(\nabla\rho)\,\d\mm.
\]
Thus if $(X,\sfd,\mm)$ is infinitesimally strictly convex, the set $\partial^-E(\rho)$ contains at most one element. The conclusion then follows from the weak-strong closure of $\partial^-E$.
\end{proof}
We now turn to geodesics: we will discuss only the case of \emph{infinitesimally Hilbertian} spaces, although weak $C^1$ regularity is valid on more general circumstances (see \cite{Gigli13}). We recall that $(X,\sfd,\mm)$ is infinitesimally Hilbertian provided
\[
\|f+g\|^2_{\s^2}+\|f-g\|^2_{\s^2}=2\|f\|^2_{\s^2}+2\|g\|^2_{\s^2},\qquad\forall f,g\in\s^2(X),
\]
and that on infinitesimally Hilbertian spaces  we  have
\[
D^-f(\nabla g)=D^+f(\nabla g)=D^-g(\nabla f)=D^+g(\nabla f),\qquad\mm-a.e. \ \forall f,g\in\s^2(X),
\]
so that in particular infinitesimally Hilbertian spaces are infinitesimally strictly convex. The common value of the above expressions will be denoted by $\nabla f\cdot\nabla g$.

The proof of weak $C^1$ regularity is based on the following lemma, proved in \cite{Gigli13}:
\begin{lemma}[`Weak-strong' convergence]\label{le:weakstrong}
Let $(X,\sfd,\mm)$ be an infinitesimally Hilbert space. Also:
\begin{itemize}
\item[i)]  Let  $(\mu_n)\subset\probt X$  a sequence with uniformly bounded densities, such that letting $\rho_n$ be the density of $\mu_n$ we have and  $\rho_n\to \rho$ $\mm$-a.e. for some probability density  $\rho$. Put $\mu:=\rho \mm$.
\item[ii)] Let $(f_n) \subset \s^2(X)$ be such that:
\[
{\sup}_{n\in \mathbb{N}} \int |Df_n|^2\,\d\mm <\infty,
\]
and assume that $f_n \to f$ $\mm$-a.e. for some Borel function $f$.

\item[iii)] Let $(g_n) \subset \s^2(X)$ and $g \in \s^2(X)$ such that $g_n \to g$ $\mm$-a.e. as $n \rightarrow +\infty$ and:
\[
\sup_{n\in \mathbb{N}} \int |D g_n|^2\,\d\mm <\infty,\qquad\qquad \lim_{n\rightarrow \infty} \int |D g_n|^2 \, \d\mu_n = \int |Dg|^2\,\d\mu.
\]
\end{itemize}
Then
\[
\mathop{\lim}_{n\rightarrow \infty}\int \nabla f_n\cdot \nabla g_n \,d\mu_n =\int \nabla f\cdot \nabla g \,d\mu.
\]
\end{lemma}
{We then have the following result:}
\begin{proposition}[Weak $C^1$ regularity for geodesics]\label{prop:geod2}
With the same assumptions of Proposition \ref{prop:geod} assume furthermore that $(X,\sfd,\mm)$ is infinitesimally Hilbertian and that for the densities $\rho_t$ of $\mu_t$ we have that $\rho_s\to\rho_t$ in  $L^p(X,\mm)$ { for some, and thus any, $p\in[1,\infty)$,} as $s\to t$.

Then $(\mu_t)$ is a weakly $C^1$ curve.
\end{proposition}
\begin{proof}
By Proposition \ref{prop:geod}, its proof and taking into account the assumption of infinitesimal Hilbertianity we know that for every $t\in[0,1)$  and $f\in L^1\cap \s^2(X)$ we have
\begin{equation}
\label{eq:dergeo}
\lim_{h\downarrow0}\frac{\int f\,\d\mu_{t+h}-\int f\,\d\mu_t}h=\int\nabla f\cdot\nabla\phi_t\,\d\mu_t.
\end{equation}
To conclude it is enough to show that the right hand side of the above expression is continuous in $t$. Pick $t\in[0,1)$ and let $(t_n)\subset [0,1]$ be a sequence converging to $t$. Up to pass to a subsequence, not relabeled, and using the assumption of strong convergence in $L^1(X,\mm)$ of $\rho_{t_n}$ to $\rho_t$, we can assume that $\rho_{t_n}\to\rho_t$ $\mm$-a.e. as $n\to\infty$. The proof of Proposition \ref{prop:geod} grants that $\sup_{n}\|\phi_{t_n}\|_{\s^2}<\infty$ and that by point $(ii)$ of Proposition \ref{prop:HL} we know that $\phi_{t_n}(x)\to\phi_t(x)$ as $n\to \infty$ for every $x\in X$. Finally, it is obvious that $\lim_{n\to\infty}\int |Df|^2\,\d\mu_{t_n}=\int |Df|^2\,\d\mu_t$ (because weak convergence in duality with $C_b(X)$ plus uniform bound on the density grant convergence of the densities in all the $L^p$'s, $p<\infty$ and weak convergence in duality with $L^1(X,\mm)$).

Thus by Lemma \ref{le:weakstrong} we deduce the desired continuity of the right hand side of \eqref{eq:dergeo} for $t\in[0,1)$. Continuity at $t=1$ is obtained by considering the geodesic $t\mapsto\mu_{1-t}$.
\end{proof}

It is worth recalling that the assumptions of Proposition \ref{prop:geod2} are fulfilled on $\RCD(K,\infty)$ spaces when $(\mu_t)$ is a (in fact `the') geodesic connecting two measures with bounded support and bounded density (see \cite{Gigli13}).

\def\cprime{$'$}


\begin{thebibliography}{10}

\bibitem{AmbrosioColomboDimarino12}
{\sc L.~Ambrosio, M.~Colombo, and S.~Di~Marino}, {\em Sobolev spaces in metric
  measure spaces: reflexivity and lower semicontinuity of slope}.
\newblock Submitted Paper, 2012.

\bibitem{AmbrosioGigli11}
{\sc L.~Ambrosio and N.~Gigli}, {\em A user's guide to optimal transport}.
\newblock Modelling and Optimisation of Flows on Networks, Lecture Notes in
  Mathematics, Vol. 2062, Springer, 2011.

\bibitem{AmbrosioGigliSavare08}
{\sc L.~Ambrosio, N.~Gigli, and G.~Savar{\'e}}, {\em Gradient flows in metric
  spaces and in the space of probability measures}, Lectures in Mathematics ETH
  Z\"urich, Birkh\"auser Verlag, Basel, second~ed., 2008.

\bibitem{AmbrosioGigliSavare11-2}
\leavevmode\vrule height 2pt depth -1.6pt width 23pt, {\em Metric measure
  spaces with riemannian {R}icci curvature bounded from below}.
\newblock Accepted by Duke math. J. ArXiv:1109.0222, 2011.

\bibitem{AmbrosioGigliSavare11}
\leavevmode\vrule height 2pt depth -1.6pt width 23pt, {\em Calculus and heat
  flow in metric measure spaces and applications to spaces with {R}icci bounds
  from below}, Inventiones mathematicae,  (2013), pp.~1--103.

\bibitem{BenamouBrenier00}
{\sc J.-D. Benamou and Y.~Brenier}, {\em A computational fluid mechanics
  solution to the {M}onge-{K}antorovich mass transfer problem}, Numer. Math.,
  84 (2000), pp.~375--393.

\bibitem{Bogachev07}
{\sc V.~I. Bogachev}, {\em Measure theory. {V}ol. {I}, {II}}, Springer-Verlag,
  Berlin, 2007.

\bibitem{Gigli12}
{\sc N.~Gigli}, {\em On the differential structure of metric measure spaces and
  applications}.
\newblock Accepted by Mem. Amer. Math. Soc.. ArXiv:1205.6622, 2012.

\bibitem{Gigli13}
\leavevmode\vrule height 2pt depth -1.6pt width 23pt, {\em The splitting
  theorem in non-smooth context}.
\newblock Preprint, arXiv:1302.5555, 2013.

\bibitem{Gigli-Kuwada-Ohta10}
{\sc N.~Gigli, K.~Kuwada, and S.-i. Ohta}, {\em Heat flow on {A}lexandrov
  spaces}, Communications on Pure and Applied Mathematics, 66 (2013),
  pp.~307--331.

\bibitem{Heinonen07}
{\sc J.~Heinonen}, {\em Nonsmooth calculus}, Bull. Amer. Math. Soc. (N.S.), 44
  (2007), pp.~163--232.

\bibitem{Lisini07}
{\sc S.~Lisini}, {\em Characterization of absolutely continuous curves in
  {W}asserstein spaces}, Calc. Var. Partial Differential Equations, 28 (2007),
  pp.~85--120.

\bibitem{Otto01}
{\sc F.~Otto}, {\em The geometry of dissipative evolution equations: the porous
  medium equation}, Comm. Partial Differential Equations, 26 (2001),
  pp.~101--174.

\bibitem{Villani09}
{\sc C.~Villani}, {\em Optimal transport. Old and new}, vol.~338 of Grundlehren
  der Mathematischen Wissenschaften, Springer-Verlag, Berlin, 2009.

\end{thebibliography}
\end{document}